\documentclass[journal]{IEEEtran}

\usepackage{amsmath, amssymb, algorithm, amsfonts, bbm}
\usepackage[noend]{algpseudocode}
\usepackage[american]{circuitikz}
\usepackage[shortlabels]{enumitem}
\usepackage{mathtools, mathrsfs}
\usepackage{tikz}
\usepackage{pgfplots}

\ctikzset{bipoles/length=1.4cm}

\DeclareMathOperator{\rank}{rank}
\DeclareMathOperator{\im}{im}
\DeclareMathOperator{\spn}{span}
\DeclareMathOperator{\col}{col}

\DeclareMathOperator{\tr}{tr}
\DeclareMathOperator*{\argmin}{arg\,min}

\usetikzlibrary{arrows, automata, patterns, calc, decorations.pathmorphing, decorations.markings, shapes.misc}

\newcommand{\bbR}{\mathbb{R}}

\newcommand{\bbN}{\mathbb{N}}

\newcommand{\bbS}{\mathbb{S}}

\newcommand{\bR}[1]{\bbR^{#1}}

\newcommand{\calF}{\mathcal{F}}
\newcommand{\calG}{\mathcal{G}}
\newcommand{\calH}{\mathcal{H}}

\newcommand{\calV}{\mathcal{V}}

\newcommand{\calW}{\mathcal{W}}

\newcommand{\X}{\bbR^n}
\newcommand{\BX}{\bbR^{n\times n}}
\newcommand{\BUX}{\bbR^{n\times m}}
\newcommand{\U}{\bbR^m}

\newcommand{\h}{\mathfrak{h}}
\newcommand{\f}{\mathfrak{f}}
\newcommand{\g}{\mathfrak{g}}
\newcommand{\hp}{\mathfrak{H}}

\newcommand{\partialf}{\partial^{\textrm{f}}}
\newcommand{\partials}{\partial^{\textrm{s}}}

\newcommand{\F}{\mathbf{F}}
\renewcommand{\H}{\mathbf{H}}
\newcommand{\G}{\mathbf{G}}

\newcommand{\B}{\mathscr{B}}

\newcommand{\abs}[1]{\left\lvert#1\right\rvert}
\newcommand{\bbm}{\begin{bmatrix*}}
	\newcommand{\ebm}{\end{bmatrix*}}

\newcommand{\inv}{^{-1}}
\renewcommand{\t}{^\top}
\renewcommand{\a}{^\ast}

\newcommand{\half}{\frac{1}{2}}
\renewcommand{\vec}{\mathrm{vec}}

\newcommand{\pdx}[1]{\frac{\partial #1}{\partial x}}

\renewcommand{\d}{\mathrm{d}}

\newcommand{\norm}[1]{\lVert#1\rVert}
\newcommand{\snorm}[1]{\lVert#1\rVert}
\newcommand{\set}[2]{\left\{ #1 \,\left|\, \vphantom{#1} #2 \right. \right\}}

\newcommand{\calQ}{\mathcal{Q}}

\newcommand{\inner}[2]{\left\langle#1, #2\right\rangle}
\newcommand{\sinner}[2]{\langle#1, #2\rangle}

\renewcommand{\epsilon}{\varepsilon}

\newcommand{\HS}{\mathscr{H}}

\renewcommand{\epsilon}{\varepsilon}

\newtheorem{theorem}{Theorem}
\newtheorem{lemma}[theorem]{Lemma}
\newtheorem{proposition}[theorem]{Proposition}
\newtheorem{corollary}[theorem]{Corollary}
\newtheorem{remark}{Remark}
\newtheorem{definition}{Definition}

\title{Kernel-based identification of nonlinear port-Hamiltonian systems}
\author{Brayan M. Shali
	\thanks{B. M. Shali is with the Department of Electrical Engineering (ESAT), KU Leuven, Leuven, Belgium. Email: \texttt{brayan.shali@kuleuven.be}}
	\quad and \quad 
	Henk J. van Waarde 
	\thanks{H. J. van Waarde is with the Bernoulli Institute for Mathematics, Computer Science, and Artificial Intelligence, University of Groningen, Groningen, The Netherlands. Email: \texttt{\texttt{h.j.van.waarde@rug.nl}}}
	\thanks{B. M. Shali acknowledges financial support from the European Research Council under the Advanced ERC Grant Agreement SpikyControl (n.101054323).}
	\thanks{H. J. van Waarde acknowledges financial support from the Dutch Research Council under the NWO Talent Programme Veni Agreement (VI.Veni.22.335).} }

\pgfplotsset{
	every axis/.append style={
		label style={font=\footnotesize},
		tick label style={font=\footnotesize},
		legend style={font=\footnotesize},
		title style={font=\footnotesize},
	}
}

\begin{document}

\maketitle

\begin{abstract}
	Port-Hamiltonian systems provide a structured framework for modeling physical systems by explicitly capturing their energy storage, dissipation, and exchange. However, deriving such models often requires detailed physical insight and precise knowledge of system parameters, which may not be available in practice. In this paper, we propose a kernel-based framework for the identification of port-Hamiltonian systems from input-state-output data. In contrast to conventional parametric approaches, the maps defining the port-Hamiltonian system are represented in suitably chosen reproducing kernel Hilbert spaces. This leads to an infinite-dimensional optimization problem over the corresponding function spaces. Our main result establishes a representer theorem that reduces this problem to a tractable finite-dimensional one. Since the reduced problem is non-convex, we further provide an algorithm for its solution and prove its convergence.
\end{abstract}

\section{Introduction}

Port-Hamiltonian systems provide a unifying framework to model real-world physical systems by capturing their energy storage, dissipation, and exchange in a structured way \cite{vanderchaft2014}. Their modular, energy-based formulation makes them especially powerful for representing interconnected multi-domain systems while guaranteeing consistency with physical laws \cite{jeltsema2009}. At the core of port-Hamiltonian models lies the property of passivity, which guarantees their stability and enables robust control design through energy shaping and damping injection \cite{doerfler2009, vanderschaft2017}. However, deriving such models often requires deep insight into the system’s physics and precise knowledge of its parameters, which can be challenging in practice. In this context, a data-driven approach can play a complementary role by enabling the estimation of relevant structures and parameters directly from data. We take such an approach in this paper by proposing a framework for identifying nonlinear port-Hamiltonian systems from input-state-output data.

Identification of port-Hamiltonian systems has received considerable attention in recent years. Most contributions in the control community focus on the linear case, e.g., \cite{cherifi2019, benner2020, schwerdtner2021, morandin2023, ortega2024a}. The approach in \cite{cherifi2019} consists of two steps: an unstructured system is identified from input-output data and subsequently approximated by a port-Hamiltonian system. In \cite{benner2020} and \cite{schwerdtner2021}, the system is identified directly from frequency-response data, with \cite{schwerdtner2021} additionally accounting for measurement noise. In \cite{morandin2023}, direct identification from input-state-output data is achieved via dynamic mode decomposition \cite{kutz2016} under the assumption that the Hamiltonian is known \emph{a priori}. Finally, \cite{ortega2024a} provides a comprehensive treatment of the single-input-single-output case.

In contrast to these works, we consider the identification of \emph{nonlinear} port-Hamiltonian systems and adopt a \emph{kernel-based} approach. Kernel-based methods have emerged as a popular framework for system identification \cite{pillonetto2010, pillonetto2014, dinuzzo2015, ljung2020}. In this framework, identification is formulated as a function or operator estimation problem in a reproducing kernel Hilbert space (RKHS) \cite{paulsen2016}. Kernel-based methods are particularly well suited for modeling nonlinear functions due to their expressive power and flexibility \cite{micchelli2006, lazar2024}. Furthermore, although the resulting estimation problems are typically infinite-dimensional, computational tractability is retained due to the celebrated representer theorem, which reduces them to finite-dimensional problems depending only on the data.

While kernel-based methods are well suited for identifying nonlinear models, accurate interpolation of the observed data alone does not guarantee that the resulting models are physically interpretable or accurate outside the observed operating regime. This motivates the incorporation of physical properties and structure into the model and identification procedure. In the kernel-based literature, several works incorporate properties such as stability \cite{pillonetto2018}, positivity \cite{grussler2017,khosravi2019}, (incremental) dissipativity \cite{vanwaarde2022},  and fading memory \cite{huo2026}. Related ideas also appear in our previous work \cite{shali2024}, where nonnegativity of the input-output operator is imposed as a necessary condition for passivity, and in \cite{ebrahimkhani2025}, which considers kernel-based estimation of the frequency response of a strictly passive linear system.

In this work, we incorporate a port-Hamiltonian structure directly into a kernel-based identification framework. This presents several challenges. First, the Hamiltonian must be identified from data on its gradient rather than from direct function evaluations. Second, the dissipative structure must satisfy a nonnegativity constraint, which is inherently nonlinear and therefore cannot be handled within standard RKHS formulations. To address these challenges, we model the Hamiltonian and input maps in suitable RKHSs, while the interconnection and dissipation structure are modeled in a sum-of-squares fashion \cite{marteau-ferey2020, muzellec2022}. Our main contributions are twofold. First, we establish a representer theorem for the resulting infinite-dimensional estimation problem, allowing it to be reduced to a finite-dimensional one. Second, since the reduced problem is nonconvex, we propose a solution algorithm and prove its convergence.

The relevance of the proposed framework extends beyond system identification to control design. Models that accurately reproduce observed data may still fail to capture properties that are important for controller synthesis, which has motivated a growing interest in so-called \emph{identification for control} \cite{gevers2005}. A related line of work considers the direct synthesis of controllers from data, thereby bypassing the identification step altogether, see \cite{hu2023, huang2024, dejong2025} for examples in the kernel-based setting. Although control is not considered explicitly in the present paper, the proposed identification framework produces port-Hamiltonian models that are passive by construction. This enables the use of a broad range of passivity-based control techniques and provides a natural starting point for robust control design. As an example, \cite{wang2025} employs a kernel-based framework for passive controllers, where passivity guarantees closed-loop stability while the controller parameters are optimized for performance.

Our work is also related to recent developments in physics-informed machine learning \cite{karniadakis2021}, where learning methods are designed to respect underlying physical structure. In particular, there has been significant interest in incorporating Hamiltonian and port-Hamiltonian structures \cite{cherifi2020}. Examples include Gaussian-process models of Hamiltonian systems \cite{bertalan2019} and kernel-based approaches for learning Hamiltonian vector fields \cite{smith2024}. A particularly influential development is the introduction of Hamiltonian neural networks (HNNs) \cite{greydanus2019}, which have since been extended to account for external inputs \cite{zhong2020} and dissipation \cite{holmsen2024}. Building on these developments, various learning methods for full port-Hamiltonian systems have been proposed in \cite{desai2021, beckers2022, moradi2025, cherifi2025}. In contrast to the present work, these methods rely on finite-dimensional parameterizations, whereas we adopt a kernel-based formulation.

The remainder of this paper is organized as follows. Section~\ref{sec:notation_preliminaries} introduces notation and relevant preliminaries. Section~\ref{sec:problem} formulates the identification problem as an optimization problem. Section~\ref{sec:representer} presents our first main result, namely, a representer theorem for this optimization problem. Section~\ref{sec:algorithm} presents our second main result, namely, an algorithm for solving the resulting finite-dimensional problem and a proof of its convergence. Section~\ref{sec:example} illustrates the results with a simple example, while Section~\ref{sec:conclusion} concludes the paper. The appendix contains several technical results from RKHS theory that are used throughout the paper.

%\note{Leftover papers: \cite{moreschini2026} moment matching using RKHS, it is about model reduction, maybe not so relevant}

\section{Notation and preliminaries}\label{sec:notation_preliminaries}

We use mostly standard notation. We denote the set of positive integers by $\bbN$, set of real numbers by $\bbR$, and the set of nonnegative real numbers by $\bbR_+$. For $n\in\bbN$, we denote the set $\{1,\dots,n\}$ by $[n]$. We denote the Cartesian product of sets $X$ and $Y$ by $X\times Y$, and the $n$-times Cartesian product of $X$ with itself by $X^n$. We denote the set of real $n\times p$ matrices by $\bbR^{n\times p}$, and the set of real $n\times n$ symmetric matrices by $\bbS^n$. We denote the smallest eigenvalue of a symmetric matrix $X\in\bbS^n$ by $\lambda_{\min}(X)$. We denote the matrix obtained by vertically concatenating the matrices $X_i\in\bbR^{n_i\times p}$, $i\in[N]$, by $\col(X_1,\dots,X_N)\in\bbR^{(\sum_{i=1}^{N}n_i)\times p}$.

Consider a differentiable map $H: \bbR^n\to\bbR$. We denote the partial derivative of $H$ with respect to the $i$-th component by $\partial_i H:\bbR^n\to\bbR$, and the vector of partial derivatives (gradient) by $\partial H:\bbR^n\to\bbR^n$. Consider a differentiable map of two variables $\h: \bbR^n\times\bbR^n \to \bbR$. We denote the partial derivative of $\h$ with respect to the $i$-th component of the first variable by $\partialf_i \h: \bbR^n\times\bbR^n \to \bbR$, and the gradient with respect to the first variable by $\partialf \h: \bbR^n \times\bbR^n \to \bbR^n$. Similarly, we denote  the partial derivative of $\h$ with respect to the $i$-th component of the second variable by $\partials_i \h: \bbR^n\times\bbR^n \to \bbR$, and  the gradient with respect to the second variable by $\partials \h: \bbR^n\times\bbR^n \to \bbR^n$.
	
\subsection{Hilbert spaces}
Given a Hilbert space $\calW$, we denote the inner product by $\inner{\cdot}{\cdot}_{\calW}$ and the induced norm by $\norm{\cdot}_{\calW}$. We often omit the subscript when the space is clear from the context. Given another Hilbert space $\calV$, we denote the Banach space of bounded linear operators from $\calW$ to $\calV$ by $\B(\calW,\calV)$ and the corresponding operator norm by $\norm{\cdot}$. We use the shorthand notation $\B(\calW) = \B(\calW, \calW)$. We denote the adjoint of $Q\in\B(\calW,\calV)$ by $Q^* \in \B(\calV,\calW)$. The operator $Q\in\B(\calW)$ is \emph{nonnegative} if $\inner{Gw}{w} \geq 0$ for all $w\in\calW$. We denote the subset of nonnegative operators by $\B(\calW)_+\subset\B(\calW)$. Given an operator $Q\in\B(\calW)$, we write $Q\geq 0$ to indicate that $Q$ is nonnegative and self-adjoint, i.e, $Q = Q\a\in \B(\calW)_+$. In the case where $\calW = \bbR^n$ and, thus, $\B(\calW) = \bbR^{n\times n}$, this means that $Q \geq 0$ if $Q$ is symmetric positive semidefinite. Note that an operator $Q\in\B(\calW)$ is nonnegative if and only if $Q+Q\a\geq 0$. Given self-adjoint operators $A,B\in\B(\calW)$, we write $A\geq B$ or $B\leq A$ to indicate that $A-B\geq 0$.

Suppose that $\calW$ is separable and the vectors $w_i$, $i\in\bbN$, form an orthonormal basis. We say that $Q\in\B(\calW)$ is a Hilbert-Schmidt operator if its Hilbert-Schmidt norm
\begin{equation}
	\norm{Q}_{\HS}^2 = \sum_{i=1}^{\infty} \norm{Qw_i}^2_\calW
\end{equation}
is finite. We denote the subset of Hilbert-Schmidt operators by $\HS(\calW) \subset \B(\calW)$, and the subset of of nonnegative Hilbert-Schmidt operators by $\HS(\calW)_+ \subset \HS(\calW)$.

\subsection{Matrix-valued maps and convexity}

Consider a differentiable map $A:\bbR^n\to\bbR^{p\times p}$. We denote the partial derivative of $A$ with respect to the $i$-th component by $\partial_iA:\bbR^n\to\bbR^{p \times p}$. The gradient of $A$ is the operator-valued map $\partial A: \bbR^{n}\to \B(\bbR^n,\bbR^{p\times p})$ given by
\begin{equation}
	\partial A(x)y = \sum_{i=1}^n\partial_{i}A(x)y_{i}
\end{equation}
for all $x,y\in\bbR^{n}$, where $y_i$ is the $i$-th component of $y$. The \emph{linearization} of $A$ around $\bar x\in \bbR^n$ is given by the map
\begin{equation}
	x\mapsto A(\bar x) + \partial A(\bar x)(x-\bar x)
\end{equation}
We say that $A$ has a \emph{pointwise} property if $A(x)$ has this property for all $x\in\bbR^n$, e.g., $A$ is pointwise symmetric if $A(x)$ is symmetric for all $x\in\bbR^n$. A pointwise symmetric $A$ is \emph{positive semidefinite convex (psd-convex)} if
\begin{equation}
	A((1-\lambda)x_1 + \lambda x_2) \leq (1-\lambda)A(x_1) + \lambda A(x_2)
\end{equation}
for all $x_1,x_2\in\bbR^n$ and $\lambda\in[0,1]$. Examples of psd-convex maps include $A(x) = x\t x$ and $A(x) = x x\t$. We say that $A$ is \emph{psd-concave} if $-A$ is psd-convex. Note that affine maps are both psd-convex and psd-concave, the finite sum of psd-convex maps is psd-convex, and the composition of a psd-convex map and an affine map is psd-convex. If $A$ is psd-convex, then its linearization around $\bar x\in\bbR^n$ is an underestimation of $A$, i.e.,
\begin{equation}
	A(\bar x) + \partial A(\bar x)(x-\bar x) \leq A(x)
\end{equation}
for all $x\in \bbR^n$, see \cite[Lemma~2.2]{dinh2012}. If $p=1$, i.e., $A$ is scalar-valued, then the definition of psd-convexity coincides with the usual definition of convexity. In this case, we say that $A$ is \emph{strongly convex} with \emph{convexity parameter} $\rho > 0$ if the map $x\mapsto A(x) - \half \rho\norm{x}^2$ is convex. Finally, we note that the definitions above extend to any matrix-valued map $A: \calV\to \bbR^{p\times p}$ where $\calV$ is a finite-dimensional Hilbert space and, thus, isometrically isomorphic to $\bbR^n$. In particular, they extend to any $\calV$ that can be written as the Cartesian product of spaces of the form $\bbR^{n\times m}$ and spaces of the form $\bbR^n$.

\section{Problem formulation}\label{sec:problem}

We consider the problem of identifying a (nonlinear) port-Hamiltonian system from input-state data. In particular, we consider an input-state-output port-Hamiltonian system
\begin{subequations}\label{eq:sys_pH}
	\begin{align}
		\dot x &= (J(x) - R(x)) \partial H(x) + G(x) u, \label{eq:sys_pH_dyn} \\
		y &= G(x)\t \partial H(x),
	\end{align}
\end{subequations}
with input $u\in\bbR^m$, state $x\in\bbR^n$, and output $y\in\bbR^m$. Here, $J:\X\to\BX$ is a pointwise skew-symmetric map that describes the lossless interconnection structure, while $R:\X\to\BX$ is a pointwise symmetric and nonnegative map that describes the dissipative resistive structure. The map $H:\bbR^n\to\bbR$ is referred to as the \emph{Hamiltonian} and represents the stored energy in the system, while $G:\bbR^n\to\bbR^{n\times m}$ is the input (port) map that describes how the external input enters the state dynamics and how the output is paired with the effort variables of the energy-storing elements. It is well-known that port-Hamiltonian systems are cyclo-passive in general and passive if the Hamiltonian $H$ is bounded from below, see \cite[Section~6.1]{vanderschaft2017} for details. 

In the context of this paper, it is useful to reformulate the dynamics \eqref{eq:sys_pH_dyn} as
\begin{equation}\label{eq:pH_dynamics}
	\begin{aligned}
		\dot x &= -F(x) \pdx H(x) + G(x) u,
	\end{aligned}
\end{equation}
where $F:\X\to\BX$ is pointwise nonnegative. To see that the two formulations are equivalent, note that 
\begin{equation}
	F(x) = -J(x) + R(x)
\end{equation}
is nonnegative if $J(x) = -J(x)\t$ and $R(x) \geq 0$. Conversely, if $F(x)$ is nonnegative, then
\begin{align}
	J(x) = \frac{1}{2}\left(F(x)\t - F(x)\right),\\
	R(x) = \frac{1}{2}\left(F(x)\t + F(x)\right),
\end{align}
are such that $J(x)\t = -J(x)$ and $R(x) \geq 0$. 

Now, suppose that we are given a data set
\begin{equation}\label{eq:data}
	(x_i, \dot x_i, u_i) \in \X\times\X\times\U, \quad i\in[N],
\end{equation}
consisting of $N\in\bbN$ samples of the state, state derivative and input of an unknown port-Hamiltonian system. At a high-level, we consider the problem of identifying a system of the form \eqref{eq:pH_dynamics} that fits the data well. In other words, we want to find $H:\X\to\bbR$, $G:\X\to\BUX$ and pointwise nonnegative $F:\X\to\BX$ such that the \emph{data misfit}
\begin{equation}\label{eq:misfit}
	L(F,H,G) = \sum_{i=1}^{N}\norm{\dot x_i + F(x_i)\partial H(x_i) - G(x_i)u_i}^2
\end{equation}
is small. We formalize what small means by introducing an appropriate cost functional in the next section.

For simplicity of exposition, we assume that $m = 1$ so that $G:\bbR^n\to\bbR^n$. The case where $m >1$ follows by considering the components of $u$ as separate inputs such that
\begin{equation}
	G(x) u = \sum_{i=1}^{m} G_i(x) u^i,
\end{equation}
where $u^i$ is the $i$-th component of $u$ and $G_i:\bbR^n\to\bbR^n$. We also note that we can also incorporate data on the output by augmenting the data misfit. In particular, if we also have $N$ samples of the output $y_i\in\bbR^m$, $i\in[N]$, then the augmented data misfit contains the additional term 
\begin{equation}
	\sum_{i=1}^{N}\norm{y_i - G(x_i)\t \partial H(x_i)}^2.
\end{equation}
Further details can be found in Remark~\ref{rem:output_data} and Remark~\ref{rem:output_data_algorithm}.

\section{Representer theorem}\label{sec:representer}

In this section, we provide a partial solution to the problem stated in the last section. We do this by leveraging the theory of reproducing kernel Hilbert spaces. First, we formalize the problem statement by formulating an appropriate infinite-dimensional constrained optimization problem. Then, as the main result in this section, we obtain a representer theorem for the latter that reduces it to a finite-dimensional constrained optimization problem. Although finite-dimensional, this problem is nonconvex and, thus, nontrivial to solve. Therefore, the next section is devoted to providing an algorithm to compute a solution.

Our approach to solving the problem leverages the theory of reproducing kernel Hilbert spaces. The relevant background on this theory can be found in the appendix. In particular, we model $H:\X\to\bbR$ as a map from a reproducing kernel Hilbert space $\calH$ with kernel $\h:\X\times\X\to\bbR$, and $G:\bbR^n\to\bbR^n$ as a map from a reproducing kernel Hilbert space $\calG$ with kernel $\g: \bbR^n\times\bbR^n\to\BX$. We make no special assumptions on the spaces $\calH$ and $\calG$, or, equivalently, the kernels $\h$ and $\g$. However, we note that the particular choice of kernels matters as it (partially) determines the \emph{subclass} of port-Hamiltonian systems that are used to fit the data. For example, we can ensure that the identified system is passive by choosing a kernel $\h$ which guarantees that the identified map $H$ is bounded from below, see Remark~\ref{rem:passivity}.

In the meantime, note that we cannot naively model a pointwise nonnegative map $F:\X\to\BX$ as an arbitrary member of a reproducing kernel Hilbert space. This is because pointwise nonnegative maps are not preserved under negative scalar multiplication, hence no vector space can consist of only pointwise nonnegative maps. Instead, we take a different approach inspired by the sum-of-squares idea introduced in \cite{marteau-ferey2020} how the output is paired with the energy variables, and further developed in \cite{shali2024} for operator-valued kernels, see Appendix~\ref{app:nonnegative} for details. 

To this end, consider a feature map $\phi:\X\to\B(\X,\calW)$, where $\calW$ is a Hilbert space, and let $\f: \X\times\X\to\B(\X)$ be the corresponding kernel, i.e., $\f(x,y) = \phi(x)\a\phi(y)$. We model a pointwise nonnegative map $F:\X\to\BX$ as
\begin{equation}\label{eq:model_F}
	F = \phi(\cdot)\a Q \phi(\cdot),
\end{equation}
where $Q\in\HS(\calW)$ is nonnegative, that is, $Q\in\HS(\calW)_+$. The map $F$ is pointwise nonnegative because
\begin{equation}
	\inner{F(x)\xi}{\xi}_{\X} = \inner{Q\phi(x)\xi}{\phi(x)\xi}_\calW \geq 0
\end{equation}
for all $x,\xi\in\X$, where we used the fact that $Q$ is nonnegative.

\begin{remark}
	The properties of the model classes considered above depend strongly on the choice of kernel. The problem of selecting an appropriate kernel is beyond the scope of this paper. Nevertheless, we briefly recall several notions of universality that characterize the approximation capabilities of the resulting model classes. A kernel (or feature map) is called \emph{universal} if any continuous function can be approximated to arbitrary precision on compact sets by functions from its associated RKHS \cite{micchelli2006}. A kernel is called \emph{differentially universal} if any continuously differentiable function and its derivatives can be approximated in the same sense by functions from its associated RKHS and their derivatives \cite{guella2022}. The Gaussian kernel is a well-known example of a differentially universal kernel. Moreover, the model class \eqref{eq:model_F} is universal for pointwise nonnegative continuous maps whenever the feature map is universal, that is, any continuous pointwise nonnegative map can be approximated in the sense above by a map of the form \eqref{eq:model_F}. This is shown in \cite{marteau-ferey2020} for scalar-valued maps, and in \cite{muzellec2022} for symmetric-valued maps.
\end{remark}

With the models of $F$, $H$ and $G$ established, our goal is to find $(Q,H,G)\in\HS(\calW)_+\times\calH\times\calG$ such that the data misfit $L(F,H,G)$ is small, where $F$ is given by \eqref{eq:model_F}. More precisely, we consider the regularized least squares problem
\begin{equation}\label{eq:min_QHg}
	\begin{aligned}
		\min \quad&
		L(F,H,G) + \alpha\norm{Q}_{\HS}^2 + \beta \norm{H}_\calH^2 + \gamma\norm{G}_\calG^2\\
		\text{s.t.}\quad & (Q,H,G)\in\HS(\calW)_+\times\calH\times\calG\text{ and \eqref{eq:model_F} holds},
	\end{aligned}
\end{equation}
where $\alpha >0$, $\beta >0$ and $\gamma >0$. Here, the first term promotes small misfit $L(F,H,G)$, i.e., a good fit of the data. The other terms promote small norms $\norm{Q}_{\HS}$, $\norm{H}_\calH$ and $\norm{G}_\calG$, which is interpreted as low complexity of the maps $F$, $H$ and $G$, respectively. The parameters $\alpha$, $\beta$ and $\gamma$ regulate the balance between good fit of the data and model complexity.

Note that \eqref{eq:min_QHg} is an infinite-dimensional constrained optimization problem, which is difficult (if at all possible) to solve directly. Fortunately, the following representer theorem reduces it to a \emph{finite-dimensional} problem and, thus, provides a path to a tractable solution.

\begin{theorem}\label{thm:representer_main}
	Problem \eqref{eq:min_QHg} admits a minimizer of the form
	\begin{subequations}\label{eq:representer_QFHg}
		\begin{alignat}{2}
			Q &= \sum_{i,j=1}^{N} \phi(x_i) M_{ij} \phi(x_j )\a,\qquad&M_{ij}\in\bbR^{n\times n}, \label{eq:representer_Q_main}\\ 
			H &= \sum_{i=1}^{N} \partials \h(\cdot,x_i)\t c_{i},\qquad &c_i\in\bbR^n, \label{eq:representer_H_main} \\
			G & = \sum_{i=1}^{N} \g(\cdot, x_i) v_i,\qquad & v_i\in\bbR^n \label{eq:representer_g_main}
		\end{alignat}
	\end{subequations}
	where the block matrix 
	\begin{equation}\label{eq:representer_M}
		M = \bbm M_{11}  & \cdots & M_{1N}\\ \vdots & \ddots & \vdots \\ M_{N1} & \cdots & M_{NN} \ebm\in\bbR^{nN\times nN}
	\end{equation}
	is nonnegative. Furthermore, the map
	\begin{equation}
		F = \sum_{i,j=1}^{N} \f(\cdot, x_i)M_{ij} \f(\cdot, x_j )\t, \label{eq:representer_F_main}
	\end{equation}
	is such that $F = \phi(\cdot)\a Q\phi(\cdot)$.
\end{theorem}

\begin{IEEEproof}
	Let $\bar \calQ\subset\HS(\calW)_+$ be the subset of operators of the form \eqref{eq:representer_Q_main}, where the block matrix in \eqref{eq:representer_M} is nonnegative. Similarly, let $\bar\calH\subset\calH$ and $\bar\calG\subset\calH$ be the subspaces of maps of the form \eqref{eq:representer_H_main} and \eqref{eq:representer_g_main}, respectively. Moreover, let the map $c:\HS(\calW)_+\times\calH\times\calG\to\bbR_+$ be the cost of \eqref{eq:min_QHg}, that is,
	\begin{equation*}
		c(Q,H,G) = L(F,H,G) + \alpha\norm{Q}_{\HS}^2 + \beta \norm{H}_\calH^2 + \gamma\norm{G}_\calG^2
	\end{equation*}
	where $F = \phi(\cdot)\a Q \phi(\cdot)$. Let $(Q_0,H_0,g_0)\in\bar \calQ\times\bar \calH\times\bar \calG$ be arbitrary and let $\bar\calQ_0$,  $\bar\calH_0$ and $\bar\calG_0$ be the closed and bounded subsets of $Q\in\bar \calQ$, $H\in\bar\calH$ and $G\in\bar \calG$, respectively, such that $\alpha\norm{Q}_{\HS} \leq c_0$, $\beta\norm{H}_\calH \leq c_0$ and $\gamma\norm{G}_\calG \leq c_0$, where
	\begin{equation}
		c_0 = c(Q_0,H_0,g_0).
	\end{equation}
	We will show that the problem
	\begin{equation}\label{eq:min_QHg_bar_0}
		\begin{aligned}
			\min \quad&
			c(Q,H,G)\\
			\text{s.t.}\quad & (Q,H,G)\in\bar \calQ_0\times\bar \calH_0\times\bar \calG_0
		\end{aligned}
	\end{equation}
	has a minimizer that is also a minimizer of \eqref{eq:min_QHg}. 
	
	First, note that $c$ is nonnegative and continuous in $(Q,H,G)$. Moreover, $\bar \calQ_0\times\bar \calH_0\times\bar \calG_0$ is a closed and bounded subset of a finite-dimensional vector space, hence it is compact. By the extreme value theorem, this implies that \eqref{eq:min_QHg_bar_0} has a minimizer. Let $(\bar Q_0, \bar H_0, \bar g_0)$ be a minimizer of \eqref{eq:min_QHg_bar_0} and let $(Q,H,G)\in\HS(\calW)_+\times\calH\times\calG$ be arbitrary. Our goal is to show that
	\begin{equation}\label{eq:cbar0_c}
	c(\bar Q_0, \bar H_0, \bar g_0) \leq c(Q,H,G).
	\end{equation}
	In view of Lemma~\ref{lem:representer_G}, Lemma~\ref{lem:representer_H}, and Lemma~\ref{lem:representer_F}, there exist $(\bar Q, \bar H, \bar G)\in \bar \calQ\times\bar \calH\times\bar \calG$ such that
	\begin{equation}\label{eq:cbar_c}
		c(\bar Q, \bar H, \bar G) \leq c(Q,H,G)
	\end{equation}
	If $c(\bar Q, \bar H, \bar G) \leq c_0$, then $\alpha\norm{\bar Q}_{\HS} \leq c_0$, $\beta\norm{\bar H}_\calH \leq c_0$ and $\gamma\norm{\bar G}_\calG \leq c_0$, hence $(\bar Q, \bar H, \bar G)\in \bar \calQ_0\times\bar \calH_0\times\bar \calG_0$ and, thus,
	\begin{equation}\label{eq:cbar0_cbar_1}
		c(\bar Q_0, \bar H_0, \bar g_0)\leq c(\bar Q, \bar H, \bar G)
	\end{equation}
	by definition of $(\bar Q_0, \bar H_0, \bar g_0)$. If $c_0 < c(\bar Q, \bar H, \bar G)$, then 
	\begin{equation}\label{eq:cbar0_cbar_2}
		c(\bar Q_0, \bar H_0, \bar g_0)< c(\bar Q, \bar H, \bar G)
	\end{equation}
	because $(Q_0,H_0,g_0)\in \bar \calQ_0\times\bar \calH_0\times\bar \calG_0$, hence 
	\begin{equation}
		c(\bar Q_0, \bar H_0, \bar g_0)\leq c(Q_0, H_0, g_0) = c_0.
	\end{equation}
	Combining \eqref{eq:cbar_c}, \eqref{eq:cbar0_cbar_1} and \eqref{eq:cbar0_cbar_2} yields \eqref{eq:cbar0_c}, which shows that $(\bar Q_0, \bar H_0, \bar g_0)$  is indeed a minimizer of \eqref{eq:min_QHg}. Note that $(\bar Q_0, \bar H_0, \bar g_0)\in \bar \calQ\times\bar \calH\times\bar \calG$, hence  $\bar Q_0, \bar H_0$ and $\bar g_0$ are of the form \eqref{eq:representer_Q_main}, \eqref{eq:representer_H_main} and \eqref{eq:representer_g_main}, respectively, where the block matrix in \eqref{eq:representer_M} is nonnegative. Finally, \eqref{eq:representer_F_main} follows from the fact that $\phi(\cdot)\a \phi(x_i) = \f(\cdot,x_i)$ for all $i\in[N]$.
\end{IEEEproof}

Theorem~\ref{thm:representer_main} implies that we can solve \eqref{eq:min_QHg} by substituting \eqref{eq:representer_QFHg} and solving for the coefficients $M_{ij}\in\bbR^{n\times n}$, $c_{i}\in\bbR^n$, $v_i\in\bbR^n$, $i,j\in[N]$, under the constraint that the block matrix in \eqref{eq:representer_M} is nonnegative. To this end, let
\begin{equation}
	\F=\col(\F_1,\dots,\F_N)\in\bbR^{nN\times nN}
\end{equation}
be the Gram matrix associated with $\f$, 
\begin{equation}
	\H = \col(\H_1,\dots,\H_N)\in\bbR^{nN\times nN}
\end{equation}
be the partial derivative Gram matrix associated with $\h$, and
\begin{equation}
	\G = \col(\G_1,\dots,\G_N)\in\bbR^{nN\times nN}
\end{equation}
be the Gram matrix associated with $\g$, where all three Gram matrices are associated with the data $x_i\in\bbR^n, i\in[N]$. We refer to \eqref{eq:gram_Gi}, \eqref{eq:gram_Hi} and \eqref{eq:gram_Fi} in the appendix for the definitions of these Gram matrices. Here, we use them to obtain the following corollary of Theorem~\ref{thm:representer_main}.
\begin{corollary}\label{cor:representer}
	Problem \eqref{eq:min_QHg} has a minimizer of the form \eqref{eq:representer_QFHg} which can be found by solving
	\begin{equation}\label{eq:min_Mvc}
		\begin{aligned}
			\min \quad&
			\bar L(M,c,v) + \alpha\tr(\F M\F M\t)+ \beta c\t \H c + \gamma v\t \G v\\
			\text{s.t.}\quad & M\in\bbR^{nN\times nN}_+,\ c\in\bbR^{nN},\ v\in\bbR^{nN}
		\end{aligned}
	\end{equation} 
	where
	\begin{equation}
		\bar L(M,c,v) =  \sum_{i=1}^{N}\norm{\dot x_i + \F_i M \F_i\t \H_i c - u_i\G_i v}^2
	\end{equation}
\end{corollary}
\begin{IEEEproof}
	Let $Q\in\HS(\calW)_+$, $H\in \calH$ and $G\in\calG$ be of the form in Theorem~\ref{thm:representer_main}. As shown in the appendix, it follows that
	\begin{equation}
		G(x_i) = \G_i v,\qquad \norm{G}^2_\calG = v\t \G v
	\end{equation}
	see \eqref{eq:representer_eval_norm} and the preceding discussion. Similarly,
	\begin{equation}
		\partial H(x_i) = \H_i c,\qquad \norm{H}_\calH = c\t \H c
	\end{equation}
	see \eqref{eq:representer_H_eval} and Lemma~\ref{lem:representer_H_norm}, while
	\begin{equation}
		F(x_i) = \F_i M \F_i\t,\qquad \norm{Q}_{\HS}^2 = \tr(\F M\F M\t)
	\end{equation}
	see \eqref{eq:representer_F_eval} and Lemma~\ref{lem:representer_F_norm}. This implies that
	\begin{equation}
		L(F,H,G) = \bar L(M,c,v)
	\end{equation}
	and the original problem \eqref{eq:min_QHg} reduces to \eqref{eq:min_Mvc}.
\end{IEEEproof}

Although problem \eqref{eq:min_Mvc} is finite-dimensional, it is not convex because of  the (nonlinear) terms $\F_i M\F_i\t \H_i c$, $i\in[N]$, in $\bar L(M,c,v)$. This means that solving \eqref{eq:min_Mvc} is nontrivial and requires additional effort.

\begin{remark}\label{rem:output_data}
	We can also incorporate data of the output of the system. Suppose that, in addition to \eqref{eq:data}, we have $N$ samples of the output $y_i\in\bbR$, $i\in[N]$. To incorporate this, we add the term $\sum_{i=1}^{N}\norm{y_i - G(x_i)\t \partial H(x_i)}^2$ to the data misfit $L(F,H,G)$. It is straightforward to show that Theorem~\ref{thm:representer_main} remains valid, while the cost in \eqref{eq:min_Mvc} acquires the additional term $\sum_{i=1}^{N}\norm{y_i - v\t \G_i\t \H_i c}^2$. The remaining results can also be adapted to incorporate output data, see Remark~\ref{rem:output_data_algorithm}.
\end{remark}

\begin{remark}\label{rem:passivity}
	Port-Hamiltonian systems are cyclo-passive, i.e., the internal energy of the system (the Hamiltonian) does not increase on closed trajectories of the state. More precisely, if the state trajectory is such that $x(t_0) = x(t_1)$ for some $t_0,t_1\in\bbR$, then $H(x(t_0)) \leq H(x(t_1))$. If the Hamiltonian is bounded from below, then the system is passive, i.e., the internal energy of the system does not increase on \emph{any} trajectory of the state. In view of Theorem~\ref{thm:representer_main}, the Hamiltonian obtained by solving problem \eqref{eq:min_QHg} is bounded from below if the partial derivative kernel sections $\partials_j\h(\cdot,x)$ are bounded from below for all $x\in\bbR^n$ and $j\in[n]$. In other words, we can ensure that the identified system is passive by choosing a kernel $\h$ whose partial derivative kernel sections are bounded from below. This is the case for, e.g., the Gaussian kernel.
\end{remark}

\section{Optimization algorithm}\label{sec:algorithm}

In this section, we present an algorithm that provides a tractable solution to problem \eqref{eq:min_Mvc} and, thus, to the problem considered in this paper. The algorithm follows the general approach of \cite{dinh2012}, where the authors consider a class of (nonconvex) optimization problems with convex cost and psd-convex-concave matrix inequality constraints. The key idea is to solve the problem by successively solving subproblems obtained by linearizing the concave part of the constraints. This renders each subproblem convex and therefore tractable using standard convex optimization techniques. Although \cite[Theorem~4.3]{dinh2012} establishes convergence of this algorithm under a set of technical assumptions, these are difficult to verify for the problem considered here. Instead, we provide a direct convergence proof based on \cite[Theorem~3.1]{meyer1976}.

\subsection{Problem reformulation}

We first show that \eqref{eq:min_Mvc} can be reformulated as a problem with a convex cost and psd-convex-concave matrix inequality constraints.  The only nonconvex term in the cost of \eqref{eq:min_Mvc} is $\bar L(M,c,v)$. We can move this term to the constraints by introducing appropriate auxiliary variables $p_i\in\bbR$, $i\in[N]$. In particular, problem \eqref{eq:min_Mvc} is equivalent to
\begin{equation}\label{eq:min_Mvcp}
	\begin{aligned}
		\min \quad&
		\sum_{i=1}^{N}p_i^2 + \alpha\tr(\F M\F M\t)+ \beta c\t \H c + \gamma v\t \G v\\
		\text{s.t.}\quad & \norm{\dot x_i + \F_i M \F_i\t \H_i c - u_i\G_i v}^2 \leq p_i^2,\ i\in[N],\\ 
		& M+M\t \geq 0
	\end{aligned}
\end{equation}
where, for ease of notation, we have omitted the constraints $M\in\bbR^{nN\times nN}$, $c\in\bbR^{nN}$, $v\in\bbR^{nN}$, $p_i\in\bbR$, $i\in[N]$. The $i$-th constraint is equivalent to
\begin{equation}\label{eq:ith_constraint}
	\bbm p_iI & \dot x_i + \F_i M \F_i\t \H_i c - u_i\G_i v \\ \ast & p_i \ebm \geq 0
\end{equation}
by the Schur complement characterization of positive semidefiniteness, hence \eqref{eq:min_Mvcp} is an optimization problem with convex cost and matrix inequality constraints. 

The constraint \eqref{eq:ith_constraint} is biaffine and can therefore be written as a psd-convex-concave matrix inequality. To see this, define the linear map $X_i: \bbR^{nN\times nN} \to \bbR^{(n+1)\times (n+1)}$ by
\begin{equation}\label{eq:Xi}
	X_i(M) = \bbm \F_i M\t \F_i \t & 0 \\ 0 & 0 \ebm,
\end{equation}
the linear map $Y_i: \bbR^{nN} \to\bbR^{(n+1)\times (n+1)}$ by
\begin{equation}\label{eq:Yi}
	Y_i(c) = \bbm 0 & \H_i c \\ 0 & 0 \ebm,
\end{equation}
and the affine map $Z_i:\bbR^{nN} \to \bbS^{n+1}$ by
\begin{equation}\label{eq:Zi}
	Z_i(v) = \bbm 0 & \dot x_i - u_i\G_i v \\ \dot x_i \t - u_i v\t \G_i\t & 0 \ebm.
\end{equation}
With these definitions, the constraint \eqref{eq:ith_constraint} can be written as
\begin{equation}\label{eq:pi_XiYi}
	X_i(M)\t Y_i(c) + Y_i(c)\t X_i(M) + Z_i(v) +p_i I \geq 0.
\end{equation}
Note that
\begin{equation}
	X\t Y + Y\t X = X\t X + Y\t Y - (X - Y)\t(X-Y)
\end{equation}
for all $X,Y\in \bbR^{(n+1)\times (n+1)}$, and the maps
\begin{align}
	(X,Y)&\mapsto X\t X + Y\t Y, \label{eq:psd_convex_1}\\ (X,Y)&\mapsto (X - Y)\t(X-Y) \label{eq:psd_convex_2}
\end{align}
are psd-convex, see, e.g., \cite[Lemma~3.1]{dinh2012}. 

Now, define the map $A_i:\bbR^{nN\times nN}\times \bbR^{nN} \to \bbS^{n+1}$ by
\begin{equation}\label{eq:Ai}
	A_i(M,c) = X_i(M)\t X_i(M) + Y_i(c)\t Y_i(c)
\end{equation}
and the map $B_i: \bbR^{nN\times nN}\times \bbR^{nN} \to \bbS^{n+1}$ by
\begin{equation}\label{eq:Bi}
	B_i(M,c) = (X_i(M) - Y_i(c))\t(X_i(M)-Y_i(c))
\end{equation}
Note that $A_i$ and $B_i$ are psd-convex since $X_i$ and $Y_i$ are linear, and the maps in \eqref{eq:psd_convex_1} and \eqref{eq:psd_convex_2} are psd-convex. Moreover, \eqref{eq:pi_XiYi} can be written as 
\begin{equation}\label{eq:AB_constraint}
	B_i(M,c) - A_i(M,c) - Z_i(v) - p_iI \leq 0
\end{equation}
and, thus, problem \eqref{eq:min_Mvcp} is equivalent to
\begin{equation}\label{eq:min_SDP_AiBi}
	\begin{aligned}
		\min \quad&
		f(p,M,c,v) \\
		\text{s.t.}\quad & M+M\t \geq 0 \text{ and }\eqref{eq:AB_constraint} \text{ holds for all }i\in[N]
	\end{aligned}
\end{equation}
where $f:\bbR^N\times \bbR^{nN\times nN}\times \bbR^{nN}\times\bbR^{nN}\to \bbR$ is given by
\begin{equation*}
	f(p,M,c,v) = p\t p + \alpha\tr(\F M \F M\t) + \beta c\t \H c + \gamma v\t \G v,
\end{equation*}
and $p \in \bbR^N$ is the vector whose $i$-th entry is $p_i$. Note that $f$ is convex because $\F$, $\H$ and $\G$ are positive semidefinite, see Remark~\ref{rem:tr_convex}. Since $Z_i$ is affine, it is also psd-convex. Therefore, $A_i(M,c) + Z_i(v) + p_iI$ is psd-convex and \eqref{eq:AB_constraint} defines a psd-convex-concave matrix inequality constraint.

\subsection{Iterative algorithm}

We solve the nonconvex problem \eqref{eq:min_SDP_AiBi} by an iterative algorithm in which each step consists of solving a convex subproblem obtained by linearizing the concave part of the constraints. Specifically, let $(p^k,M^k,c^k,v^k)$ denote the $k$-th iterate of the algorithm. The linearization of $A_i$ around this point is given by
\begin{equation*}
	A^k_i (M,c) = A_i(M^k, c^k) + \partial A_i(M^k, c^k)(M- M^k, c - c^k)
\end{equation*}
It is straightforward to verify that
\begin{multline*}
	\hspace{-3mm}\partial A_i(M^k,c^k)(M ,c) = X_i(M)\t X_i(M^k) + X_i(M^k)\t X_i(M)\\ + Y_i(c)\t Y_i(c^k) + Y_i(c^k)\t Y_i(c),
\end{multline*}
hence $\partial A_i(M^k,c^k)(M^k ,c^k) = 2A_i(M^k, c^k)$ and, thus,
\begin{equation}\label{eq:Aik}
	A^k_i (M,c) = -A_i(M^k, c^k) + \partial A_i(M^k, c^k)(M, c).
\end{equation}
The $k$-th subproblem is obtained by replacing $A_i$ in \eqref{eq:AB_constraint} with its linearization $A^k_i$, resulting in the convex constraint
\begin{equation}\label{eq:AB_constraint_k}
	B_i(M,c) - A_i^k(M,c) - Z_i(v) - p_iI \leq 0.
\end{equation}
The latter is equivalent to the affine matrix inequality
\begin{equation}\label{eq:AB_constraint_k_linear}
\bbm A_i^k(M,c) + Z_i(v) +p_iI  & X_i(M)\t-Y_i(c)\t \\ X_i(M) -Y_i(c) & I \ebm \geq 0
\end{equation}
by the Schur complement characterization of negative semidefiniteness. The next iterate $(p^{k+1},M^{k+1},c^{k+1},v^{k+1})$ is obtained by solving the (convex) $k$-th subproblem
\begin{equation}\label{eq:min_SDP_AiBi_k}
	\begin{aligned}
		\min \quad&
		f(p,M,c,v)\\
		\text{s.t.}\quad & M+M\t \geq 0 \text{ and }\eqref{eq:AB_constraint_k_linear} \text{ holds for all }i\in[N].
	\end{aligned}
\end{equation}
The constraints in this subproblem are affine matrix inequalities, whereas the cost function is convex quadratic. Indeed, in view of Remark~\ref{rem:tr_convex}, the cost $f(p,M,c,v)$ can be written as
\begin{equation*}
	\bbm p \\ \vec(M) \\ c \\ v \ebm\t \bbm I & 0 & 0 & 0 \\ 0 & \alpha\F\otimes\F & 0 & 0 \\ 0 & 0 & \beta\H & 0 \\ 0 & 0 & 0 & \gamma\G \ebm \bbm p \\ \vec(M) \\ c \\ v \ebm.
\end{equation*}
where we recall that $\H$, $\G$, $\F$ and, thus, $\F\otimes\F$ are positive semidefinite. Consequently, the $k$-th subproblem \eqref{eq:min_SDP_AiBi_k} can be reformulated as a linear semidefinite program.
\begin{remark}\label{rem:strictly_feasible}
	The $k$-th subproblem is easily seen to be strictly feasible. Indeed, $M + M\t > 0$ for $M = I$, while the inequality \eqref{eq:AB_constraint_k_linear}, equivalently \eqref{eq:AB_constraint_k}, can be made strict by choosing $p_i > 0$ sufficiently large, $i\in[N]$. Consequently, since $f$ is convex, continuous and bounded from below, the $k$-th subproblem has at least one minimizer for any $(p^k,M^k,c^k,v^k)$. In particular, this means that  the iteration is well-defined for any initial guess $(p^0, M^0, c^0, v^0)$. If, in addition, $f$ is \emph{strongly} convex, i.e., $\F$, $\H$ and $\G$ are invertible, then the $k$-th subproblem has a \emph{unique} minimizer for any $(p^k,M^k,c^k,v^k)$.
\end{remark}

Although not strictly necessary, we may choose the initial guess $(p^0, M^0, c^0, v^0)$ to be a strictly feasible point of the original problem \eqref{eq:min_SDP_AiBi}, i.e., $M^0+(M^0)\t > 0$  and 
\begin{equation}
	B_i(M^0,c^0) - A_i(M^0,c^0) - Z_i(v^0) - p_i^0 I< 0 
\end{equation}
for all $i\in[N]$. The latter is equivalent to
\begin{equation}
	\norm{\dot x_i + \F_i M^0 \F_i\t \H_i c^0 - u_i\G_i v^0}^2 < p_i^0,
\end{equation}
for all $i\in[N]$, hence we can take, e.g., 
\begin{equation}
	M^0 = I,\ c^0 = 0,\ v^0 = 0,\ p_i^0 = \norm{\dot x_i}^2 + 1,\ i\in[N].
\end{equation}
This leads to the following summary of the iterative algorithm for solving problem \eqref{eq:min_SDP_AiBi}.

\begin{algorithm}
	\caption{Solving problem \eqref{eq:min_SDP_AiBi}}
	\label{alg:consistency_implementation}
	\begin{algorithmic}[1]
		\State Define $X_i$, $Y_i$, $Z_i$, $A_i$ and $B_i$ according to \eqref{eq:Xi}, \eqref{eq:Yi}, \eqref{eq:Zi}, \eqref{eq:Ai} and \eqref{eq:Bi}, respectively, for all $i\in[N]$.
		\State Choose $(p^0, M^0, c^0, v^0)$ such that $M^0+(M^0)^\top>0$ and
		\begin{equation*}
			\norm{\dot x_i + \F_i M^0 \F_i\t \H_i c^0 - u_i\G_i v^0}^2 < p_i^0
		\end{equation*}
		for all $i\in[N]$.
		\State Initialize $k = 0$ and choose a tolerance $\epsilon > 0$.
		\State Define $A_i^k$ according to \eqref{eq:Aik} for all $i\in[N]$.
		\State Solve \eqref{eq:min_SDP_AiBi_k} to obtain $(p^{k+1}, M^{k+1}, c^{k+1}, v^{k+1})$.
		\State Stop if
		\begin{equation*}
			\norm{(p^{k+1}, M^{k+1}, c^{k+1}, v^{k+1}) - (p^{k}, M^{k}, c^{k}, v^{k})}<\epsilon.
		\end{equation*}
 		\State Set $k = k+1$ and return to Step~4.
	\end{algorithmic}\label{alg:algorithm}
\end{algorithm}

\begin{remark}\label{rem:output_data_algorithm}
	We can modify Algorithm~\ref{alg:algorithm} to incorporate output data. Building on the modifications in Remark~\ref{rem:output_data}, we introduce auxiliary variables $\hat p_i\in\bbR$, $i\in[N]$, so that the modified problem \eqref{eq:min_Mvcp} contains the additional term $\sum_{i=1}^N \hat p_i^2$ in the cost, as well as the additional constraints $\norm{y_i - v\t \G_i\t \H_i c}^2 \leq \hat p_i^2$,\\ $i\in[N]$. Although these constraints are not convex, they can be expressed as biaffine matrix inequalities, which can be reformulated as psd-convex–concave matrix inequalities. To this end, let $\hat X_i(v),\hat Y_i(c)\in\bbR^{(n+1)\times 2}$, $i\in[N]$, be given by 
	\begin{equation}
		\hat X_i(v) = \bbm \G_i v & 0 \\ 0 & 0 \ebm,\quad \hat Y_i(c) = \bbm 0 & \H_i v \\ 0 & 0 \ebm,
	\end{equation}
	and let $\hat Z\in\bbR^{2\times 2}$, $i\in[N]$, be given by
	\begin{equation}
		\hat Z_i= \bbm 0& -y_i \\ -y_i & 0\ebm
	\end{equation}
	Then, the additional constraints in \eqref{eq:min_Mvcp} are equivalent to
	\begin{equation}
		\hat X_i(v)\t \hat Y_i(c) + \hat Y_i(c)\t \hat X_i(v) + \hat Z_i + \hat p_i I \geq 0.
	\end{equation}
	This corresponds to adding the constraints
	\begin{equation*}
		\bbm \hat A_i^k(v,c) + \hat Z_i + \hat p_iI  & \hat X_i(v)\t-\hat Y_i(c)\t \\ \hat X_i(v) -\hat Y_i(c) & I \ebm \geq 0,\ i\in[N],
	\end{equation*}
	to the $k$-th subproblem \eqref{eq:min_SDP_AiBi_k}, where
	\begin{equation}
		\hat A^k_i(v,c) = -\hat A_i(v^k, c^k) + \partial \hat A_i(v^k, c^k)(v, c)
	\end{equation}
	is the linearization of 
	\begin{equation}
		\hat A_i(v,c) = \hat X_i(v)\t \hat X_i(v) + \hat Y_i(c)\t \hat Y_i(c)
	\end{equation}
	whose derivative is given by
	\begin{multline*}
		\partial \hat A_i(v^k,c^k)(v ,c) = \hat X_i(v)\t \hat X_i(v^k) + \hat X_i(v^k)\t \hat X_i(v)\\ + \hat Y_i(c)\t \hat Y_i(c^k) + \hat Y_i(c^k)\t \hat Y_i(c).
	\end{multline*}
	Naturally, we also need to add the term $\sum_{i=1}^N \hat p_i^2$ to the cost of the $k$-th subproblem.
\end{remark}

\subsection{Convergence analysis}

We now prove that the sequence $(p^k,M^k,c^k,v^k)$, $k\in\bbN$, generated by Algorithm~\ref{alg:algorithm} converges to a stationary point of \eqref{eq:min_SDP_AiBi}. We do this under the assumption that the Gram matrices $\F$, $\H$ and $\G$ are invertible and, thus, positive definite. This ensures that the cost function $f$ is \emph{strongly} convex with convexity parameter
\begin{equation}
	\rho_f = \min\left(1, \alpha\lambda_{\min}(\F)^2, \beta\lambda_{\min}(\H), \gamma\lambda_{\min}(\G)\right) > 0.
\end{equation}
where we used the fact that $\lambda_{\min}(\F\otimes\F) = \lambda_{\min}(\F)^2$. In particular, this means that the $k$-th subproblem \eqref{eq:min_SDP_AiBi_k} has a \emph{unique} minimizer for any $(p^k,M^k,c^k,v^k)$, see Remark~\ref{rem:strictly_feasible}.

We prove convergence of Algorithm~\ref{alg:algorithm} by rewriting it as a fixed-point iteration and applying \cite[Theorem~3.1]{meyer1976}. Define
\begin{equation}
	\Theta = \bbR^N\times \bbR^{nN\times nN}\times \bbR^{nN}\times\bbR^{nN}
\end{equation}
and let $\theta^k = (p^{k}, M^{k}, c^{k}, v^{k})\in\Theta$ for all $k\in\{0\}\cup\bbN$. Note that problem \eqref{eq:min_SDP_AiBi} can be rewritten as
	\begin{equation}\label{eq:min_compact}
		\min_{\theta\in\Theta}\ f(\theta) \quad \text{s.t.}\quad B(\theta) - A(\theta) \leq 0,
	\end{equation}
	where the maps $A,B:\Theta\to\bbS^{(n+1)N + nN}$ are psd-convex. In particular, for $\theta = (p,M,c,v)$,  $A(\theta)$ is the block diagonal matrix with $A_i(M,c) + Z_i(v)+p_iI$, $i \in[N]$, and $M+M\t$ on the block diagonal, and $B(\theta)$ is the block diagonal matrix with $B_i(M,c)$, $i \in[N]$, and $0\in\bbR^{nN\times nN}$ on the block diagonal.

	Next, consider the set-valued map $C :\Theta\rightrightarrows \Theta$ given by
	\begin{equation*}
		C(\theta)  = \set{\vartheta\in\Theta}{B(\vartheta) - A(\theta) - \partial A(\theta)(\vartheta - \theta) \leq 0}.
	\end{equation*}
	Note that $C(\theta^k)$ coincides with the constraint set of the $k$-th subproblem \eqref{eq:min_SDP_AiBi_k}, which is strictly feasible for any $\theta^k\in\Theta$, see Remark~\ref{rem:strictly_feasible}. Therefore, $C(\theta)$ is a nonempty, closed, and convex set for all $\theta\in \Theta$, hence the map $S:\Theta \to \Theta$ given by
	\begin{equation*}
		S(\theta) = \argmin_{\vartheta\in C(\theta)} f(\vartheta) ,
	\end{equation*}
	is well-defined because $f$ is strongly convex . By construction, the iterates satisfy $\theta^{k+1} = S(\theta^k)$, that is, Algorithm 1 can be written as the fixed-point iteration generated by the map $S$.

	To apply \cite[Theorem~3.1]{meyer1976}, we first restrict the map $S$ to the nonempty compact subset
	\begin{equation}
		\bar\Theta = \set{\theta\in \Theta}{\theta\in C(\theta),\ f(\theta) \leq f(\theta^0)}.
	\end{equation}
	The condition $\theta\in C(\theta)$ is equivalent to $B(\theta) - A(\theta) \leq 0$, that is, $\theta$ is feasible for problem \eqref{eq:min_compact}. Therefore, $\bar \Theta$ consists of all feasible points of \eqref{eq:min_compact} whose cost does not exceed that of the initial iterate $\theta^0$. Note that $\bar \Theta$ is nonempty because $\theta^0$ is feasible, and compact because $f$ is strongly convex. The application of \cite[Theorem~3.1]{meyer1976} then requires that the following conditions are satisfied:
	\begin{enumerate}[C1)]
		\item $S$ is uniformly continuous on $\bar\Theta$;\label{C1}
		\item $S$ is such that $S(\theta)\in\bar\Theta$ for all $\theta\in\bar\Theta$;\label{C2}
		\item $S$ is strictly monotone with respect to $f$ on $\bar\Theta$, that is, $f(S(\theta)) < f(\theta)$ for all $\theta\in \bar\Theta$ such that $S(\theta) \neq \theta$.\label{C3}
	\end{enumerate}
	
	We first establish the last two conditions.
	\begin{lemma}\label{lem:monotone}
		The following statements hold:
		\begin{enumerate}
			\item $S(\theta) \in C(S(\theta))$ for all $\theta \in \Theta$;
			\item $f(S(\theta)) \leq f(\vartheta) - \half \rho_f\norm{\vartheta - S(\theta)}^2$ for all  $\vartheta\in C(\theta)$.
		\end{enumerate}
		Moreover, conditions \ref{C2} and \ref{C3} are satisfied.
	\end{lemma}
	\begin{IEEEproof}
		Let $\theta\in \Theta$. Since $A$ is  psd-convex, its linearization is an underestimation, hence
		\begin{equation}
			B(\vartheta) - A(\vartheta) \leq B(\vartheta) - A(\theta) - \partial A(\theta)(\vartheta - \theta),
		\end{equation}
		and, thus, $\vartheta\in C(\vartheta)$ for all $\vartheta\in C(\theta)$. Then, $S(\theta)\in C(S(\theta))$ because $S(\theta)\in C(\theta)$ by definition of $S$, i.e., the first statement holds. Now, let $\vartheta\in C(\theta)$. Note that $\vartheta_\lambda = \lambda \vartheta + (1-\lambda) S(\theta)$ satisfies $\vartheta_\lambda\in C(\theta)$ and
		\begin{multline}\label{eq:strong_convexity}
			f(S(\theta)) \leq f(\vartheta_\lambda) \leq \lambda f(\vartheta) + (1-\lambda)f(S(\theta)) \\- \half\rho_f\lambda(1-\lambda) \norm{\vartheta - S(\theta)}^2
		\end{multline}
		for all $\lambda \in [0,1]$, where the first inequality follows from the definition of $S$ and the second follows from the strong convexity of $f$. Dividing by $\lambda >0$ and rearranging yields
		\begin{equation}
	 f(S(\theta)) \leq f(\vartheta) - \half\rho_f(1-\lambda) \norm{\vartheta - S(\theta)}^2
		\end{equation}
		for all $\lambda \in (0,1]$. Taking the limit as $\lambda \to 0$ shows that the second statement holds. 
		
		Now, let  $\theta\in\bar\Theta$. This implies $\theta\in C(\theta)$ and the second statement yields $f(S(\theta))<f(\theta)$ when $S(\theta) \neq \theta$, i.e., \ref{C3} is satisfied. Moreover, the first statement yields $S(\theta)\in C(S(\theta))$, while the second yields $f(S(\theta)) \leq f(\theta) \leq f(\theta^0)$, hence $S(\theta)\in \bar\Theta$, i.e., \ref{C2} is also satisfied.
	\end{IEEEproof}
	
	Only condition \ref{C1} remains to be established. Loosely speaking, we need to show that $S(\theta')$ is ``close'' to $S(\theta)$ whenever $\theta'$ is ``close'' to $\theta$.  Our approach is as follows. We first show that there exists $\vartheta\in\Theta$ that is simultaneously in the interior of $C(\theta)$ and ``close'' to $S(\theta)$. We then show that $\vartheta$ is in $C(\theta')$ when $\theta'$ is ``close'' to $\theta$. Due to Lemma~\ref{lem:monotone}, this allows us to show that $\vartheta$ is ``close'' to $S(\theta')$, hence, by the triangle inequality, that $S(\theta')$ is ``close'' to $S(\theta)$.
	
	We begin by showing that there exists $\bar\vartheta\in\Theta$ that is in the interior of $C(\theta)$ for all $\theta \in \bar\Theta$. Let $D:\Theta \times \Theta\to \bbS^{N(n+1) + nN}$ be given by
	\begin{align}
		D(\vartheta,\theta) &= B(\vartheta) - A(\theta) - \partial A(\theta)(\vartheta - \theta)
	\end{align}
	and note that $\vartheta \in C(\theta)$ if and only if $D(\vartheta,\theta) \leq 0$. 
	\begin{lemma}\label{lem:bar_vartheta}
		There exists $\bar \vartheta\in\Theta$ such that $D(\bar\vartheta,\theta) \leq - I$ for all $\theta\in\bar\Theta$.
	\end{lemma}
	\begin{IEEEproof}
		Let $\vartheta = (p,M,c,v)$. Recall that $C(\theta^k)$ coincides with the constraint set of the $k$-th subproblem \eqref{eq:min_SDP_AiBi_k}. In particular, $D(\vartheta,\theta^k) \leq - I$ if and only if $M + M\t \geq I$ and
		\begin{equation}
			B_i(M,c) - A_i^k(M,c) - Z_i(v) - p_iI \leq -I\label{eq:AB_inequality_strict}
		\end{equation}
		for all $i\in[N]$. In the following, we treat $$\theta^k = (p^{k}, M^{k}, c^{k}, v^{k})$$ as a variable. Since $\bar\Theta$ is compact and $A^k_i(I,0)$ is a polynomial in $\theta^k$ there exists $\bar A_i\in \bbS^{n+1}$ such that $A^k_i(I,0) \geq \bar A_i$ for all $\theta^k \in \bar\Theta$. Choose $\bar p_i\in\bbR$, $i\in[N]$, large enough so that
		\begin{equation}
			B_i(I,0) - \bar A_i - Z_i(0) - \bar p_iI \leq -I
		\end{equation}
		for all $i\in[N]$, and let $\bar p\in\bbR^n$ be the vector whose $i$-th entry is $\bar p_i$. It follows that $\bar \vartheta = (\bar p, I, 0,0)$ is such that  $D(\bar \vartheta,\theta^k) \leq - I$ for all $\theta_k\in\bar\Theta$, as desired.
	\end{IEEEproof}
	
	Let $\bar\vartheta\in\Theta$ be as in Lemma~\ref{lem:bar_vartheta}. Since $\bar\Theta$ is bounded,  there exists $\mu \geq \norm{\bar\vartheta}$ such that $\norm{\vartheta} \leq \mu$ for all $\vartheta\in \bar\Theta$. Let $\Theta_\mu\subset \Theta$ be the ball of radius $\mu$. The following lemma shows that, for all $\theta\in\bar\Theta$, there exists $\vartheta\in\Theta_\mu$ that is simultaneously in the interior of $C(\theta)$ and ``close'' to $S(\theta)$.
\begin{lemma}\label{lem:dense}
	For all $\epsilon\in(0,1]$ and $\theta\in\bar\Theta$, there exists $\vartheta\in\Theta_\mu$ such that $D(\vartheta, \theta)\leq -\epsilon I$ and $\norm{\vartheta-S(\theta)} \leq 2\epsilon \mu$.
\end{lemma}
\begin{IEEEproof}
	Let $\epsilon \in (0,1]$ and $\theta\in\bar\Theta$. Since $D$ is psd-convex in the first variable, $\vartheta_\lambda = \lambda\bar \vartheta + (1-\lambda)S(\theta)$ satisfies
	\begin{align*}
		D(\vartheta_\lambda, \theta) 
		&\leq \lambda D(\bar\vartheta,\theta) + (1-\lambda)D(S(\theta),\theta)\leq -\lambda I
	\end{align*}
	for all $\lambda\in[0,1]$, where we used the fact that $D(\bar\vartheta,\theta) \leq -I$, see Lemma~\ref{lem:bar_vartheta}, and $D(S(\theta), \theta)\leq 0$ because $S(\theta)\in C(\theta)$ by definition of $S$. Since $S(\theta)\in\bar\Theta\subset\bar\Theta_\mu$, we also have that
	\begin{equation}
		\norm{\vartheta_\lambda - S(\theta)} \leq \lambda(\norm{\bar\vartheta} + \norm{S(\theta)}) \leq 2\lambda \mu
	\end{equation}
	hence we can take $\vartheta = \vartheta_\lambda\in\bar\Theta_\mu$ for $\lambda = \epsilon\in(0,1]$.
\end{IEEEproof}

We are now ready to establish condition \ref{C1}.
\begin{lemma}\label{lem:continuous}
	The map $S$ is uniformly continuous on $\bar\Theta$.
\end{lemma}
\begin{IEEEproof}
Since $f$ and $D$ are continuously differentiable, they are Lipschitz continuous on the compact subsets $\Theta_\mu$ and  $\Theta_\mu\times\Theta_\mu$, i.e., there exist $L_D>0$ and $L_f > 0$ such that
\begin{align}
	\abs{f(\theta) - f(\theta')} &\leq L_f \norm{\theta-\theta'}\label{eq:Lf}\\
	\norm{D(\vartheta,\theta) - D(\vartheta,\theta')} &\leq L_D\norm{\theta - \theta'}\label{eq:LD}
\end{align}
for all $\vartheta,\theta,\theta'\in\Theta_\mu$. Since $D$ maps to symmetric matrices, \eqref{eq:LD} implies that 
\begin{equation}
	D(\vartheta,\theta) - D(\vartheta,\theta') \leq L_D\norm{\theta - \theta'} I \label{eq:LD_new}
\end{equation}
for all $\vartheta,\theta,\theta'\in\Theta_\mu$.

Let $\epsilon > 0$ be arbitrary. Take $\delta >0$ such that $\delta L_D \leq 1$, and $\theta, \theta'\in\bar\Theta$ such that $\norm{\theta-\theta'} < \delta$. Without loss of generality, we assume that $f(S(\theta)) \leq f(S(\theta'))$. By Lemma~\ref{lem:dense}, there exists $\vartheta\in\Theta_\mu$ such that $D(\vartheta, \theta) \leq -\delta L_D I$ and
\begin{equation}\label{eq:vartheta_theta}
	\norm{\vartheta - S(\theta)} \leq 2\delta L_D \mu
\end{equation}
Consequently, we obtain
\begin{equation}
	D(\vartheta,\theta') \leq D(\vartheta,\theta) + \delta L_D I \leq 0
\end{equation}
where the first inequality follows from \eqref{eq:LD_new}. This implies that $\vartheta\in C(\theta')$ and, thus, $f(\vartheta) \geq f(S(\theta'))$. Using this, we obtain
\begin{align*}
	f(\vartheta) -  f(S(\theta')) \leq f(\vartheta) -  f(S(\theta)) \leq 2\delta\mu L_fL_D
\end{align*}
where the first inequality follows from the assumption that $f(S(\theta)) \leq f(S(\theta'))$, and the second from \eqref{eq:Lf} and \eqref{eq:vartheta_theta}. Then, Lemma~\ref{lem:monotone} implies that
\begin{equation}\label{eq:vartheta_theta'}
	\frac{1}{2}\rho_f\norm{\vartheta - S(\theta')}^2 \leq f(\vartheta)  - f(S(\theta')) \leq 2\delta \mu L_fL_D
\end{equation}
Using \eqref{eq:vartheta_theta}, \eqref{eq:vartheta_theta'} and the triangle inequality, we obtain
\begin{align}\label{eq:delta_epsilon}
	\norm{S(\theta) - S(\theta')}\leq 2\delta\mu L_D + 2\sqrt{\delta \mu \frac{L_f}{\rho_f} L_D}
\end{align}
which is smaller than $\epsilon$ for sufficiently small $\delta>0$. This shows that $S$ is uniformly continuous of $\bar \Theta$, as desired.
\end{IEEEproof}

Finally, we can prove the convergence of the algorithm.
\begin{theorem}
	Suppose that the Gram matrices $\F$, $\H$ and $\G$ are invertible. The sequence $(p^k,M^k,c^k,v^k)$, $k
	\in\bbN$, converges to a stationary point or to a continuum of stationary points of problem \eqref{eq:min_compact}.
\end{theorem}
\begin{IEEEproof}
	Due to Lemma~\ref{lem:monotone} and Lemma~\ref{lem:continuous}, the three conditions of \cite[Theorem~3.1]{meyer1976} are satisfied. It follows that all accumulation points of $\theta^k$, $k\in\bbN$, are fixed points of $S$, and $\norm{\theta^{k+1} - \theta^k} \to 0$ and $f(\theta^k) \to f(\theta^*)$ as $k\to\infty$, where $\theta^*$ is a fixed point of $S$. Furthermore, either $\theta^k$, $k\in\bbN$, converges or its accumulation points, which are fixed points of $S$, form a continuum. It remains to show that the fixed points of $S$ are stationary points of problem \eqref{eq:min_compact}. To this end, the generalized first-order conditions for problem \eqref{eq:min_compact} state that $\theta$ is a stationary point if and only if $B(\theta) - A(\theta) \leq 0$ and there exists $\Lambda \in \bbS^{(n+2)N}$ such that $\Lambda \geq 0$ and
	\begin{align}
			\partial f(\theta) + \left(\partial B(\theta) - \partial A(\theta)\right)\a \Lambda &= 0,\\
			\tr\left((B(\theta) - A(\theta))\Lambda\right) &= 0,
	\end{align}
	see \cite{dinh2012, shapiro1997}. At the same time, by definition, $S(\theta)$ is a stationary point of the problem
	\begin{equation}
		\min_{\vartheta\in\Theta} f(\vartheta)\quad \text{s.t.}\quad D(\vartheta,\theta) \leq 0,
	\end{equation}
	which is the case if and only if $D(S(\theta),\theta)\leq 0 $ and there exists $\Lambda \in \bbS^{(n+2)N}$ such that $\Lambda \geq 0$ and
	\begin{align}
		\partial f(S(\theta)) + \left(\partial B(S(\theta)) - \partial A(\theta)\right)\a \Lambda &= 0,\\
		\tr\left(D(S(\theta),\theta)\Lambda\right) &= 0,
	\end{align}
	where we used the fact that the partial derivative of $D(\vartheta,\theta)$ with respect to $\vartheta$ is given by $\partial B(\vartheta) - \partial A(\theta)$. Note that $D(\theta,\theta) = B(\theta) - A(\theta)$, hence $S(\theta) = \theta$ if and only if $\theta$ is a stationary point of problem \eqref{eq:min_compact}, as desired.
\end{IEEEproof}

\begin{remark}
	The  assumption that $f$ is strongly convex is essential in proving the convergence of Algorithm~\ref{alg:algorithm}. However, we can avoid making the assumption that $\F$, $\H$ and $\G$ are invertible by appropriately restricting the spaces in which we search for $M$, $v$ and $c$. For example, suppose that $\G$ is not invertible. Let $T_1$ be an orthogonal matrix whose columns form a basis for $\im \G$, and let $T_2$ be an orthogonal matrix whose columns form a basis for $\ker \G$. Since $\G$ is symmetric, it follows that $T_2\t \G = 0$ and $\G T_2 = 0$. Furthermore, the matrix $T = [T_1\quad T_2]$ is square and invertible, hence $v\in\bbR^{nN}$ can be written as $v = T_1v_1 + T_2 v_2$ for some $v_1\in\bbR^{\rank \G}$ and $v_2\in\bbR^{nN-\rank \G}$. Note that $v$ appears in \eqref{eq:min_SDP_AiBi} only as part of the products $v\t \G v$ and $\G v$. By writing $v$ as above, these products reduce to $v_1\t \G_1 v_1$ and $T_1\G_1 v_1$, where $\G_1 = T_1\t \G T_1$ is invertible and, thus, positive definite. Therefore, by making appropriate modifications in \eqref{eq:min_SDP_AiBi}, we can optimize over $v_1$ instead of $v$ and, thus, ensure that the cost $f$ is strongly convex in $v_1$. We can make similar modifications if the matrices $\H$ and $\F$ are also not invertible.
\end{remark}

\section{illustrative example}\label{sec:example}

In this section, we demonstrate our results with an illustrative example. Consider one-port circuit consisting of a capacitor, inductor and resistor connected in series. Suppose that the capacitor is linear, i.e., $q = Cv_C$ for some constant $C\in\bbR$, where $q$ is the charge of the capacitor and $v_C$ is the voltage across it. On the other hand, suppose that the inductor and resistor are nonlinear, i.e., $\phi = L(i)$ and $v_R = R(i) i$ for some functions $L: \bbR\to\bbR$ and $R:\bbR\to\bbR$, where $\phi$ is the magnetic flux of the inductor, $v_R$ is the voltage across the resistor, and $i$ is the current through the port. The port is described by a port-Hamiltonian system with input given by the voltage $v$ across the port, and output given by the current $i$ through the port, namely,
\begin{align*}\label{eq:example_dynamics}
	\bbm \dot q \\ \dot \phi \ebm &= \bbm 0 & 1 \\ -1 & -R(L\inv(\phi))\ebm\bbm q/C \\ L\inv (\phi) \ebm + \bbm 0 \\ 1 \ebm v,\\
	i &= \bbm 0 & 1\ebm \bbm q/C \\ L\inv (\phi) \ebm
\end{align*}
The latter is of the form \eqref{eq:pH_dynamics} with $u = v$, $x = (q,\phi)$, 
\begin{align*}
	F(q,\phi) = \bbm 0 & -1 \\ 1 & R(L\inv (\phi)) \ebm,\qquad G(q,\phi) &= \bbm 0 \\ 1 \ebm,
\end{align*}
and Hamiltonian $H$ given by the total energy
\begin{equation}
	H(q,\phi) = \frac{q^2}{2C} + \int L\inv (\phi)\, \d\phi.
\end{equation}
Here, we have implicitly assumed that $L$ is invertible, at least in some region of interest.
%\old{\begin{figure}
%	\centering
%	\ctikzset{resistors/scale=0.6, inductors/scale=0.6, capacitors/scale=0.6}
%	
%	\begin{circuitikz}[scale=0.8]
%		\draw 
%		(0,0) 
%		to[open, v<=$v$, o-o] (0,2)
%		to[short, i>=${i}$] (1.25,2)
%		to[R = $R$] (4,2) 				% Resistor 
%		to[L = $L$] (4,0)				% Inductor
%		to[C = $C$] (1.25,0)				% Capacitor
%		to[short, i>=${}$] (0,0)
%		;			
%%				\draw[dashed] (1.25,-1) -- (4.75,-1) -- (4.75,3) -- (1.25,3) -- (1.25,-1);
%	\end{circuitikz}
%	\caption{Driven RLC circuit.}
%	\label{fig:RLC_circuit}
%\end{figure}}

Suppose that true system has $C = 1$, $R(i) = 0.1 + i^2$ and $L(i) = \tanh i$. We would like to obtain approximations of the true $F$, $H$ and $G$ based on data of the form \eqref{eq:data}. To this end, we model $F$, $H$ and $G$ as described in the beginning of Section~\ref{sec:representer} with (Gaussian) kernels
\begin{equation*}
	\f(x,y) = e^{\frac{\norm{x-y}^2}{2\sigma_F^2}} I,\ \h(x,y) = e^{\frac{\norm{x-y}^2}{2\sigma_H^2}}, \ \g(x,y) = e^{\frac{\norm{x-y}^2}{2\sigma_G^2}} I,
\end{equation*}
where $I$ is the $2\times 2$ identity matrix, and $\sigma_F, \sigma_H, \sigma_G > 0$ are parameters (the standard deviations of the Gaussians). We collect $N=16$ data points by sampling a state trajectory corresponding to a 30 second piecewise constant input trajectory with 1 second intervals of random values in $[-0.1,0.1]$ The state trajectory and the state data are shown in Figure~\ref{fig:sampled_state}. The state derivative data is obtained by evaluating the vector field of the true system at the state data, but we note that we could also approximate it from the state trajectory.

The solution produced by our algorithm has a data misfit of approximately $4.2\cdot 10^{-3}$. Figure~\ref{fig:outputs} shows an example output trajectory corresponding to an input trajectory of the same type as the one used for data collection. We see that, although the misfit is fairly small, the approximated trajectory is significantly different from the true one. We attribute this to two factors. First, the optimization problem that we solve is nonconvex, hence it generally has multiple local minima. Second, the model class is overly expressive, especially when using universal kernels. Indeed, the input map $G$ can fit the data well on its own, that is, for nonzero input data, we can find $G$ such that $\dot x_i \approx G(x_i)u_i$ for all $i\in[N]$. Consequently, the proposed algorithm likely converges to a suboptimal local minimum. Neither of these factors is present if we do not impose any structure and simply model the vector field and output maps as belonging to an RKHS. And, indeed, we see that the unstructured approximation in Figure~\ref{fig:outputs} performs much better than the port-Hamiltonian one.

\begin{figure}
	\centering
	\begin{tikzpicture}
		\begin{axis}[%
			width=4cm,
			height=4cm,
			scale only axis,
			xmin=-0.27,
			xmax=0.27,
			xlabel={$q$},
			ymin=-0.27,
			ymax=0.27,
			ylabel={$\phi$},
			axis lines=box,
%			axis x line*=bottom,
%			axis y line*=left,
			]
			
			\addplot [mark=o,
			mark size=1pt,   % smaller disk
			mark options={fill=black},
			mark indices={1},
			->,
			color=black,
			forget plot]
			table[col sep=tab, x index=0, y index=1]{plot_data/x_traj.txt};
			
			\addplot [color=red, only marks, mark=x, mark options={solid, red}, line width=.75pt, forget plot]
			table[col sep=tab, x index=0, y index=1]{plot_data/x_data.txt};
			
			% Grid lines
			\foreach \x in {-0.2,-0.1,0,0.1, 0.2} {
				\addplot [gray,	opacity=0.1, thin] coordinates {(\x,-1) (\x,1)};
				\addplot [gray,	opacity=0.1, thin] coordinates {(-1, \x) (1, \x)};
			}
		\end{axis}
	\end{tikzpicture}
	\hspace{15mm}
	\caption{Sampled state trajectory corresponding to a 30 second piecewise constant input with 1 second intervals of random values in $[-0.1,0.1]$.}
	\label{fig:sampled_state}
\end{figure}
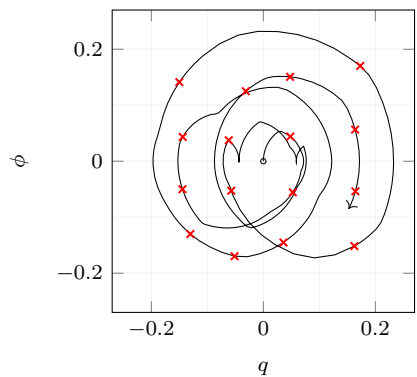
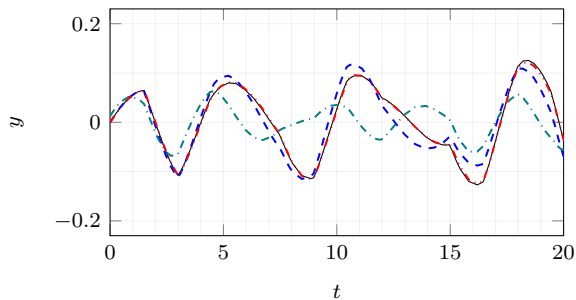
\begin{figure}
	\begin{tikzpicture}
		\begin{axis}[
			width=6cm,
			height=3cm,
			scale only axis,
			xmin=0,
			xmax=20,
			xlabel={$t$},
			ymin=-0.23,
			ymax=0.23,
			ylabel={$y$},
			axis lines=box,
%			axis x line*=bottom,
%			axis y line*=left
			]
			\addplot [color=black, forget plot]
			table[col sep=tab, x index=0, y index=1]{plot_data/y_true.txt};
			
			\addplot [color=blue, dashed, line width=.75pt, forget plot]
			table[col sep=tab, x index=0, y index=1]{plot_data/y_sunapprox.txt};
			
			\addplot [color=teal, dashdotted, line width=.75pt, forget plot]
			table[col sep=tab, x index=0, y index=1]{plot_data/y_approx_y.txt};
			
			\addplot [color=red, dashdotdotted, line width=.75pt, forget plot]
			table[col sep=tab, x index=0, y index=1]{plot_data/y_approx_ygJ.txt};
			
			% Grid lines
			\foreach \x in {0,...,30} {
				\addplot [gray,	opacity=0.1, thin] coordinates {(\x,-1) (\x,1)};
			}
			\foreach \y in {-0.2,-0.1,0,0.1, 0.2} {
				\addplot [gray,	opacity=0.1, thin] coordinates {(0, \y) (30, \y)};
			}
		\end{axis}
	\end{tikzpicture}
	\caption{Example output trajectories corresponding to a 30 second piecewise constant input with 1 second intervals of random values in $[-0.1,0.1]$: true (solid, black), unstructured approximation (dashed, blue), port-Hamiltonian approximation (dashed-dotted, teal), port-Hamiltonian approximation with additional structural assumptions (dash-dot-dotted, red). }
	\label{fig:outputs}
\end{figure}
\begin{figure}
	\centering
	\hspace{-10mm}
	\begin{tikzpicture}
		\begin{axis}[%
			align=left,
			width=6cm,
			height=3cm,
			scale only axis,
			xmin=-0.3,
			xmax=0.3,
			xlabel={$q$},
			ymin=-0.01,
			ymax=0.11,
			ylabel={$H(q,\phi)$},
			ylabel style={yshift=-3mm},
			ytick={0,0.1},
			axis lines = box,
			]
			
			\addplot [color=black, forget plot]
			table[col sep=tab, x index=0, y index=1]{plot_data/H_true1.txt};
			
			\addplot [color=teal, dashdotted, line width=.75pt, forget plot]
			table[col sep=tab, x index=0, y index=1]{plot_data/H_approx_y1.txt};
			
			\addplot [color=red, dashdotdotted, line width=.75pt, forget plot]
			table[col sep=tab, x index=0, y index=1]{plot_data/H_approx_ygJ1.txt};
			
			\addplot [color=black, forget plot]
			table[col sep=tab, x index=0, y index=1]{plot_data/H_true2.txt};
			
			\addplot [color=teal, dashdotted, line width=.75pt, forget plot]
			table[col sep=tab, x index=0, y index=1]{plot_data/H_approx_y2.txt};
			
			\addplot [color=red, dashdotdotted, line width=.75pt, forget plot]
			table[col sep=tab, x index=0, y index=1]{plot_data/H_approx_ygJ2.txt};
			
			% Grid lines
			\foreach \x in {-0.2, -0.1, 0, 0.1, 0.2} {
				\addplot [gray,	opacity=0.1, thin] coordinates {(\x,-1) (\x,1)};
			}
			\foreach \y in {-0.2,-0.1, 0, 0.1, 0.2} {
				\addplot [gray,	opacity=0.1, thin] coordinates {(-1, \y) (1, \y)};
			}
		\end{axis}
		
	\end{tikzpicture}

	\caption{Sections of the Hamiltonians: true (black, solid), approximated (teal, dash-dotted), and approximated under additional structural assumptions (red, dash-dot-dotted) at $\phi = 0$ (bottom) and $\phi = 0.2$ (top)}
	\label{fig:hamiltonians}
\end{figure}
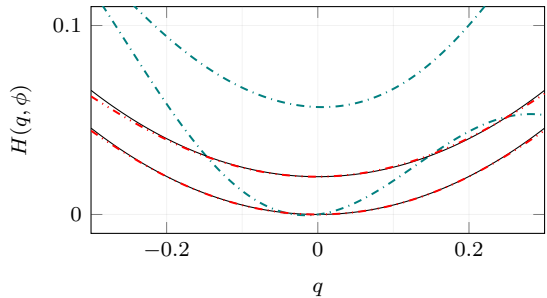

We have only imposed that the system is port-Hamiltonian, i.e., of the form \eqref{eq:pH_dynamics}. However, our framework allows us to naturally incorporate additional prior knowledge by further restricting the structure of the port-Hamiltonian system. For instance, in this circuit example, the interconnection structure implies that the skew-symmetric part of $F$ and the input map $G$ are known. Specifically, we can assume that
\begin{equation}
	F(\cdot) = \bbm 0 & 1 \\ -1 & 0 \ebm + F_s(\cdot),\quad G(\cdot) = \bbm 0 \\ 1 \ebm
\end{equation}
where $F_s:\bbR^2\to\bbR^{2\times 2}$ is pointwise \emph{symmetric} and nonnegative (resistive structure). In other words, we only need to model the symmetric part of $F$, which we do as
\begin{equation}
	F_s(\cdot) = \phi(\cdot)\a Q_s \phi(\cdot)
\end{equation}
where $Q_s\in\HS(\calW)$ is nonnegative \emph{and self-adjoint}. It is straightforward to incorporate these additional assumptions to obtain a modified representer theorem and an appropriately adapted algorithm.

The solution produced by the modified algorithm has a data misfit of approximately $2.9\cdot 10^{-6}$. In contrast to the previous port-Hamiltonian approximation, the data misfit is significantly reduced, and the approximated output trajectory in Figure~\ref{fig:outputs} closely matches the true trajectory over the entire time horizon. While the resulting approximation is not substantially more accurate than the unstructured RKHS model in terms of trajectory prediction, it achieves comparable performance while retaining a physically meaningful decomposition. This is further illustrated in Figure~\ref{fig:hamiltonians}, where the approximated Hamiltonian under the additional structural assumptions closely matches the true Hamiltonian, in contrast to the previous approximated Hamiltonian, which exhibits a significant discrepancy. The accurate recovery of the Hamiltonian is particularly valuable, as it yields an interpretable energy function that enables stability analysis and control design. In fact, since the kernel was chosen as a Gaussian, we are guaranteed that the approximated Hamiltonian is bounded from below and, thus, the approximated port-Hamiltonian system is passive.

\section{Conclusion}\label{sec:conclusion}

We proposed a kernel-based framework for the identification of nonlinear port-Hamiltonian systems from input-state-output data. We modeled the Hamiltonian and input maps in RKHSs and represented the interconnection and dissipation structure in a kernel-based sum-of-squares form. We formulated the identification problem as an infinite-dimensional optimization problem over these model classes. As our first main result, we established a representer theorem that reduces this problem to a finite-dimensional one. As our second main result, we developed an iterative solution algorithm for the resulting nonconvex optimization problem and proved its convergence.

We also illustrated our results with an example, where we saw that the flexibility of the model classes and the nonconvexity of the estimation problem can lead to suboptimal solutions. At the same time, we saw that incorporating additional prior knowledge is straightforward and can significantly improve the quality of the identified models. Future work will investigate the influence of the kernel choice on the resulting approximation and explore how the port-Hamiltonian structure of the identified models can be exploited in control-oriented applications.

\appendix
\setcounter{definition}{0}
\setcounter{theorem}{0}
\setcounter{remark}{0}

\renewcommand{\thedefinition}{\Alph{section}.\arabic{definition}}
\renewcommand{\thetheorem}{\Alph{section}.\arabic{theorem}}
\renewcommand{\theremark}{\Alph{section}.\arabic{remark}}

In this appendix, we introduce some of the theory of reproducing kernel Hilbert spaces (RKHS's) that is relevant to the discussion in this paper. One of the most attractive features of the this theory is that certain function (operator) estimation problems have computationally tractable solutions due to the \emph{representer theorem}. More specifically, the representer theorem states that the minimizers of a particular class of optimization problems over an RKHS belong to a finite-dimensional subspace of the RKHS. Therefore, these typically infinite-dimensional problems can be solved by solving an associated finite-dimensional problem.

This appendix is split into three subsections, each with a representer theorem as the main result. First we present the standard theory of RKHS's, including kernels, the reproducing property, the Gram matrix, feature maps, and the classic representer theorem. Second, we present results concerning \emph{partial derivatives} in RKHS's, namely, we prove a partial derivative reproducing property and a representer theorem in the case where the cost of the optimization problem depends on partial derivative data. Third, we present an approach to modelling pointwise nonnegative maps and prove a representer theorem in the case where the decision variable of the optimization problem is restricted to be a pointwise nonnegative map.

\subsection{Reproducing kernel Hilbert spaces}\label{app:RKHS}
We start by discussing the standard theory of RKHS's. We refer to \cite{paulsen2016} for a treatment of the case where the kernel is scalar-valued, and \cite{micchelli2005} for the extension to the general case. We also refer to the accessible technical report \cite{kalnishkan2009}. 

Consider a Hilbert space $\calG$ of operators $G:\bbR^{n}\to\bbR^{p}$. We say that $\calG$ is a \emph{reproducing kernel Hilbert space} if it admits a reproducing kernel, defined below.
\begin{definition}\label{def:RKHS}
	A map $\g:\bbR^{n}\times\bbR^{n}\to\bbR^{p\times p}$ is a \emph{reproducing kernel} (or, simply, \emph{kernel}) for $\calG$ if:
	\begin{enumerate}
		\item $\g(\cdot, x)v\in\calG$ for all $x\in\bbR^{n}$, $v\in\bbR^{p}$;
		\item $\inner{G(x)}{v}_{\bbR^{p}} = \inner{G}{\g(\cdot, x)v}_\calG$ for all $x\in\bbR^{n}$, $v\in\bbR^{p}$, and $G\in\calG$.
	\end{enumerate}
\end{definition}

The second property in Definition~\ref{def:RKHS} is referred to as the \emph{reproducing property}. The interpretation of the reproducing property is clearer in the case where the kernel is scalar-valued, i.e., $p=1$. In this case, the reproducing property states that $G(x) = \inner{G}{\g(\cdot,x)}_\calG$ for all $G\in\calG$ and $x\in\bbR^n$. The operator $\g(\cdot, x)\in\calG$ is called a \emph{kernel section}, hence the reproducing property states that the inner product of an operator and a kernel section at a given point  \emph{reproduces} the value of the operator at that point. This implies that the point evaluation functionals in an RKHS are \emph{continuous}. In fact, that latter is often used as the defining property of an RKHS.

The class of reproducing kernels is completely characterized by two properties, namely, symmetry and positive semidefiniteness, defined below.
\begin{definition}
	The map $\g:\bbR^{n}\times\bbR^{n}\to\bbR^{p\times p}$ is:
	\begin{enumerate}
		\item \emph{symmetric} if $\g(x,y)\t = \g(y,x)$ for all $x,y\in\bbR^{n}$;
		\item \emph{positive semidefinite} if \vspace{-1mm}
		\begin{equation}\label{eq:pos_ker} \vspace{-1mm}
			\sum_{i=1}^N \inner{v_i}{\g(x_i,x_j)v_j} \geq 0
		\end{equation}
		for all $N\in\bbN$, $x_i\in\bbR^{n}$ and $v_i\in\bbR^p$, $i\in[N]$.
	\end{enumerate}
\end{definition}
Any reproducing kernel needs to be symmetric and positive semidefinite due to the reproducing property. More notably, the converse also holds and is known as the Moore-Aronszajn theorem \cite{aronszajn1950}, see \cite{micchelli2005} for the extension to the general case.

The properties of symmetry and positive semidefiniteness of the map $\g:\bbR^{n}\times\bbR^{n} \to\bbR^{n\times n}$ can be expressed using the so-called \emph{Gram matrix}. Given $N\in\bbN$ and $x_i\in\bbR^{n}$, $i\in[N]$, we define the Gram matrix $\G\in\bbR^{pN\times pN}$ as
\begin{equation}\label{eq:Gram}
	\G = \bbm \g(x_1,x_1) & \cdots & \g(x_1,x_n) \\ \vdots & \ddots & \vdots \\ \g(x_n,x_1) & \cdots & \g(x_n,x_n) \ebm,
\end{equation}
Then, \eqref{eq:pos_ker} reads
\begin{equation}
	\inner{\G(v_1,\dots,v_n)}{(v_1,\dots,v_n)}_{\bbR^{pN}} \geq 0,
\end{equation}
hence the map $\g$ is symmetric and positive semidefinite if and only if, for all $N\in\bbN$ and $x_i\in\bbR^{n}$, $i\in[N]$, the corresponding Gram matrix $\G$ in \eqref{eq:Gram} is symmetric positive semi-definite.

Symmetric positive semidefinite maps  $\g:\bR{n}\times\bR{n}\to\bR{p\times p}$ can be defined via an auxiliary Hilbert space $\calW$ and a map $\phi: \bR{n}\to\B(\bR{p},\calW)$ by setting $\g(x,y) = \phi(x)^*\phi(y)$ for all $x,y\in\bbR^n$. The space $\calW$ is referred to as a \emph{feature space} and the map $\phi$ as a \emph{feature map}. The following proposition, see \cite[Theorem~7]{vanwaarde2021arxiv}, states that each symmetric positive semidefinite map has an associated feature map and vice versa.
\begin{proposition}\label{prop:feature_map}
	The map $\g:\bbR^n\times\bbR^n\to\bR{p\times p}$ is symmetric and positive semidefinite if and only if there exists a Hilbert space $\calW$ and a feature map $\phi: \bR{n}\to\B(\bR{p},\calW)$ such that $\g(x,y) = \phi(x)^*\phi(y)$ for all $x,y\in\bbR^n$.
\end{proposition}

The feature map is central in developing kernel-based algorithms and can allow for simpler analysis. Namely, it can be shown that the reproducing kernel Hilbert space $\calG$ with reproducing kernel $\g:\bbR^n\times\bbR^n\to\bR{p\times p}$ and associated feature map  $\phi: \bR{n}\to\B(\bR{p},\calW)$ consists of operators of the form $G = \phi(\cdot)^* w$, where $w$ belongs to the closure of $\spn\set{\phi(x)v}{x\in\bbR^n,v\in\bbR^p}\subset\calW$, and $\snorm{G}_\calG = \snorm{w}_\calW$. We refer to \cite[Theorem~5 and Theorem~6]{kalnishkan2009} for a proof of this statement in the case where the kernel is scalar-valued. The proof in the general case follows similar reasoning.

As mentioned at the beginning of this appendix, an attractive feature of the theory of RKHS's is that certain function estimation problems have computationally tractable solutions. In particular, let $x_i\in\bbR^n$, $i\in[N]$, and consider the  problem
\begin{equation}\label{eq:min_representer_G}
	\begin{aligned}
		\min_{G\in\calG} \quad& L(G(x_1), \dots,G(x_N)) + \alpha\norm{G}_\calG,
	\end{aligned}
\end{equation}
where $\alpha > 0$ and $L:(\bbR^p)^N\to\bbR_+$ is a continuous loss function. In the context of estimating a function from data, the term $L(G(x_1),\dots,G(x_N))$ measures the data misfit, while $\norm{G}_{\calG}$ measures the complexity of $G$ and acts as a regularizer. Although $\calG$ is typically infinite-dimensional, the minimizer of \eqref{eq:min_representer_G} belongs to a finite-dimensional subspace of $\calG$. This is the content of the representer theorem \cite[Theorem~4.2]{micchelli2005}, which follows from the following technical lemma.
\begin{lemma}\label{lem:representer_G}
	Let $x_i\in\bbR^n$, $i\in[N]$. For all $G\in\calG$, there exist $v_i\in\bbR^n$, $i\in[N]$, such that the map
	\begin{align}\label{eq:representer_G}
		\bar G &= \sum_{i=1}^{N} \g(\cdot, x_i) v_i
	\end{align}
	satisfies $\bar G(x_i) = G(x_i)$, $i\in[N]$, and $\norm{\bar G}_\calG \leq \norm{G}_\calG$.
\end{lemma}

Then, the representer theorem states the following.
\begin{proposition}\label{prop:representer}
	Let $x_i\in\bbR^n$, $i\in[N]$. Suppose that $\alpha > 0$ and $L: (\bbR^p)^N\to\bbR_+$ is continuous. Then, problem \eqref{eq:min_representer_G} has a minimizer of the form \eqref{eq:representer_G}.
\end{proposition}

As a a consequence of the representer theorem, we can solve \eqref{eq:min_representer_G} by substituting $G$ of the form \eqref{eq:representer_G}, and solving for the coefficient vectors $v_i\in\bbR^p$, $i\in[N]$. This yields a finite-dimensional optimization problem that can be compactly expressed using the Gram matrix. Namely, we can show that any operator $G$ of the form \eqref{eq:representer_G} satisfies
\begin{equation}\label{eq:representer_eval_norm}
	G(x_i) = \G_i v,\quad \norm{G}^2_\calG = v\t \G v,
\end{equation}
where $v = \col(v_1,\dots,v_N)$, $\G_i\in\bbR^{n\times nN}$ is given by
\begin{equation}\label{eq:gram_Gi}
	\G_i = \bbm \g(x_i,x_1) & \cdots & \g(x_i,x_N) \ebm
\end{equation}
for all $i\in[N]$, and $\G = \col(\G_1,\dots,\G_N)$ is the Gram matrix associated with $\g$ and $x_i\in\bbR^n$, $i\in[N]$. Consequently, a minimizer of \eqref{eq:min_representer_G} can be obtained by solving
\begin{equation}\label{eq:min_representer_G_finite}
	\min_{v\in\bbR^{pN}} \quad L(\G_1 v, \dots, \G_Nv) + \alpha v\t \G v
\end{equation}
and substituting the resulting minimizer in \eqref{eq:representer_G}.

\subsection{Partial derivatives in reproducing kernel Hilbert spaces}\label{app:derivative}

Next, we discuss an extension of the standard theory of RKHS's that deals with partial derivatives. In particular, we show that RKHS's whose kernels are (sufficiently) differentiable have a partial derivative reproducing property. This allows us to obtain a representer theorem for a class of optimization problems where the cost depends on partial derivative data. The latter is essential for the solution of the problem considered in this paper since the data misfit \eqref{eq:misfit} depends on the partial derivatives of the Hamiltonian $H$. The results in this subsection can be found in \cite{zhou2008}. The authors of \cite{zhou2008} consider operators whose domain is a \emph{compact} subset of $\bbR^n$. Here, we consider the extension to operators whose domain is the whole space $\bbR^n$. Although this extension is fairly straightforward, we still provide the relevant proofs for the sake of completeness.

Throughout this subsection, we consider an RKHS $\calH$ with a scalar-valued kernel $\h:\bbR^n\times\bbR^n\to\bbR$ (the notation is chosen to be consistent with the notation in the main body of this paper). We are interested in solving an optimization problem of the form \eqref{eq:min_representer_G}, where the cost depends on values of the \emph{partial derivatives} of an operator $H\in\calH$. For this problem to be well-posed, the operators in $\calH$ must be differentiable, which leads to the following proposition.

\begin{proposition}\label{prop:partial_reproducing}
	Suppose that $\h$ is twice continuously differentiable. Then, the following statements hold for all $i\in[n]$:		
	\begin{enumerate}
		\item $\partials_i\h (\cdot, y)\in\calH$ for all $y\in\bbR^{n}$;
		\item $\partial_i H(y)= \inner{H}{\partials_i\h(\cdot, y)}_\calH$ for all $H\in\calH$ and $y\in\bbR^{n}$.
	\end{enumerate}
\end{proposition}
\begin{IEEEproof}
	Let $i\in[N]$ and $y\in\bbR^{n}$. For each $\lambda > 0$, the map
	\begin{equation}
		h_\lambda = \frac{\h(\cdot, y + \lambda e_i)- \h(\cdot,y)}{\lambda}
	\end{equation}
	is such that $h_\lambda\in\calH$ and
	\begin{equation}\label{eq:pointwise_hk}
		\lim_{\lambda\to0} h_\lambda(x) = \partials_i\h(x,y)
	\end{equation}
	for all $x\in\bbR^n$. The reproducing property implies that
	\begin{align}
		\norm{h_\lambda}_\calH^2 &= \frac{h_\lambda(y+\lambda e_i) - h_\lambda(y)}{\lambda}.
	\end{align}
	Since $\h$ is twice continuously differentiable, it follows that
	\begin{equation}
		\lim_{\lambda\to 0}\norm{h_\lambda}_\calH^2 = \partial_ih_\lambda(y) = \partialf_i \partials_i \h(y,y),
	\end{equation}
	hence the set $\set{h_\lambda}{0<\lambda \leq \delta}\subset\calH$ is bounded for sufficiently small $\delta >0$. Every bounded sequence in a Hilbert space has a weakly convergent subsequence. This follows from the Eberlein-Smulian theorem \cite[Theorem~13.1]{conway2007}, where we note that the unit ball in $\calH$ is weakly compact because Hilbert spaces are reflexive \cite[Theorem~4.2]{conway2007}. Therefore, there exists a sequence $\lambda_k>0$, $k\in\bbN$, that converges to $0$, while the sequence $h_{\lambda_k}$, $k\in\bbN$, converges weakly to some $h\in\calH$, i.e.,
	\begin{equation}\label{eq:weak_hk}
		\lim_{k\to\infty} \inner{h_{\lambda_k}}{H}_\calH = \inner{h}{H}_\calH
	\end{equation}
	for all $H\in\calH$. Taking $H = \h(\cdot,x)$ yields 
	\begin{equation}\label{eq:weak_hkx}
		\lim_{k\to\infty} h_{\lambda_k}(x) = h(x)
	\end{equation}
	for all $x\in\bbR^n$. Since  $\lambda_k \to 0$ as $k\to\infty$, \eqref{eq:pointwise_hk} and \eqref{eq:weak_hkx} imply that $ \partials_i\h(\cdot,y) = h\in\calH$. Finally, we note that
	\begin{equation*}
		\lim_{k\to\infty}\inner{h_{\lambda_k}}{H}_\calH =  \lim_{k\to\infty}\frac{H(y+\lambda_{k}e_i) - H(y)}{\lambda_{k}} = \partial_i H(y)
	\end{equation*}
	for all $H\in\calH$, hence \eqref{eq:weak_hk} yields
	\begin{equation}
		\partial_i H(y)= \inner{\partials_i\h(\cdot, y)}{H}_\calH = \inner{H}{\partials_i\h(\cdot, y)}_\calH
	\end{equation}
	for all $H\in\calH$, which concludes the proof.
\end{IEEEproof}

Recall the defining properties of the reproducing kernel $\h$ of the RKHS $\calH$, namely, all kernel sections of $\h$ belong to $\calH$ and the reproducing property holds. Proposition~\ref{prop:partial_reproducing} tells us that analogous properties hold for partial derivatives, namely, all \emph{partial derivative} kernel sections of $\h$ belong to $\calH$ and a \emph{partial derivative} reproducing property holds. In fact, this partial derivative reproducing property implicitly shows that the operators in $\calH$ are differentiable, i.e., $\partial_i H(y)$ exists for all $i\in[N]$, $H\in\calH$, and $y\in\bbR^n$. Furthermore, it allows us to obtain a representer theorem for optimization problems where the cost depends on partial derivative data. To this end, consider the following  technical lemma.

\begin{lemma}\label{lem:representer_H}
	Let $x_i\in\bbR^n$, $i\in[N]$. For all $H\in\calH$, there exist $c_i\in\bbR^n$, $i\in[N]$, such that the map
	\begin{align}\label{eq:representer_H}
		\bar H &= \sum_{i=1}^{N} \partials\h(\cdot, x_i)\t c_i\in \calH
	\end{align}
	satisfies $\partial \bar H(x_i) = \partial H(x_i)$, $i\in[N]$, and $\snorm{\bar H}_\calH \leq \snorm{H}_\calH$.
\end{lemma}
\begin{IEEEproof}
	Consider the finite-dimensional subspace
	\begin{equation}\label{eq:bar_calH}
		\bar \calH = \spn\set{\partials_j \h(\cdot,x_i)}{i\in[N],\ j\in[n]} \subset\calH,
	\end{equation}
	and its orthogonal complement $\bar \calH^\perp\subset\calH$.
	Write $H =  \bar H + \bar H^\perp$, where $\bar H\in\bar \calH$ and $\bar H^\perp\in\bar \calH^\perp$, and note that
	\begin{equation*}
		\partial_jH(x_i) = \inner{H}{\partials_j\h(\cdot,x_i)} = \inner{ \bar H}{\partials_j\h(\cdot,x_i)} = \partial_j \bar H(x_i)
	\end{equation*}
	for all $j\in[n]$ and $i\in[N]$. Furthermore, we have that
	\begin{equation}
		\norm{H}_\calH = \snorm{\bar H}_\calH + \snorm{\bar H^\perp}_\calH,
	\end{equation}
	hence $\norm{\bar H} \leq \norm{H}$. Since $\bar H\in\bar \calH$, we can write
	\begin{equation}
		\bar H = \sum_{i=1}^{N}\sum_{j=1}^{n}\partials_j\h(\cdot,x_i)c_{ij}
	\end{equation}
	where $c_{ij}\in\bbR$ for all $i\in[N]$ and $j\in[n]$. The latter can be compactly written as in \eqref{eq:representer_H}, which concludes the proof.
\end{IEEEproof}
 
The following representer theorem is a special case of \cite[Theorem~2]{zhou2008}, which we prove for the sake of completeness.
\begin{proposition}\label{prop:representer_H}
	Let $x_i\in\bbR^n$, $i\in[N]$. Suppose that $\beta > 0$ and $L: (\bbR^n)^N\to\bbR_+$ is continuous. Then,  the problem
	\begin{equation}\label{eq:min_representer_H}
		\begin{aligned}
			\min_{H\in\calH} \quad& L(\partial H(x_1), \dots, \partial H(x_N)) + \beta\norm{H}_\calH ,
		\end{aligned}
	\end{equation}
	has a minimizer of the form \eqref{eq:representer_H}.
\end{proposition}
\begin{IEEEproof}
	Let the subspace $\bar\calH\subset\calH$ be given by \eqref{eq:bar_calH}, and the map $c:\calH\to\bbR_+$ be given by
	\begin{equation}
		c(H) = L(\partial H(x_1), \dots, \partial H(x_N)) + \beta\norm{H}_\calH.
	\end{equation}
	Let $H_0\in\bar \calH$ be arbitrary and $\bar\calH_0$ be the closed and bounded subset of operators $H\in\bar\calH$ such that $\beta \norm{H}_\calH \leq c(H_0)$. Note that $\bar\calH_0$ is compact because $\bar\calH$ is finite-dimensional. Let
	\begin{equation}
		\bar H_0 = \argmin_{H\in\bar\calH_0} c(H),
	\end{equation}
	which exists because $c$ is continuous and bounded from below, and $\bar\calH_0$ is compact. We claim that $\bar H_0$ is a minimizer of \eqref{eq:min_representer_H}. Indeed, due to Lemma~\ref{lem:representer_H}, for all $H\in\calH$, there exists $\bar H \in\bar\calH$ such that $c(\bar H) \leq c(H)$.  If $c(\bar H) \leq c(H_0)$, then, in particular, $\beta\norm{\bar H} \leq c(H_0)$, hence $\bar H\in\bar\calH_0$ and, thus, $c(\bar H_0) \leq c(\bar H)$ by definition of $\bar H_0$. On the other hand, if $c(H_0) < c(\bar H)$, then $c(\bar H_0) < c(\bar H)$ because $H_0\in\bar\calH_0$ and, thus, $c(\bar H_0) \leq c(H_0)$. It follows that $c(\bar H_0) \leq c(H)$ for all $H\in\calH$, hence $\bar H_0$ is a minimizer of \eqref{eq:min_representer_H}. Since $\bar H_0 \in \bar\calH$, it is indeed of the form \eqref{eq:representer_H}, which concludes the proof.
	\end{IEEEproof}

	Just as the classical representer theorem, the representer theorem stated in Proposition~\ref{prop:representer_H} provides a tractable solution to \eqref{eq:min_representer_H}, where a minimizer can be obtained by substituting $H$ of the form \eqref{eq:representer_H} and solving for the coefficient vectors $c_i\in\bbR^n$, $i\in[N]$. To obtain the resulting finite-dimensional problem, note that, if $H$ is of the form \eqref{eq:representer_H}, then
	\begin{equation}\label{eq:representer_H_eval}
		\partial H(x_i) = \sum_{j=1}^N \H_{ij}c_j = \H_i c,
	\end{equation}
	for all $i\in[N]$, where $\H_{ij}\in\bbR^{n\times n}$ is given by
	\begin{equation}\label{eq:Hji}
		\H_{ij} = \bbm \partialf_1\partials_1\h(x_i,x_j) & \cdots & \partialf_1\partialf_n\h(x_i,x_j)\\ \vdots & \ddots & \vdots\\  \partialf_n\partials_1\h(x_i,x_j) & \cdots & \partialf_n\partialf_n\h(x_i,x_j) \ebm
	\end{equation}
	for all $i,j\in[N]$, $c = \col(c_1,\cdots,c_n)$, and
	\begin{equation}\label{eq:gram_Hi}
		\H_i = \bbm \H_{i1} & \cdots & \H_{iN} \ebm
	\end{equation}
	for all $i\in[N]$. Let $\H = \col(\H_1,\dots,\H_N)$. We refer to $\H$ as the \emph{partial derivative Gram matrix} associated with $\h$ and $x_i\in\bbR^n$, $i\in[N]$. This is because $\H$ can be seen as the Gram matrix associated with a reproducing kernel $\hp$ obtained from the partial derivatives of $\h$, as shown below.
	\begin{proposition}\label{prop:partial_Gram}
		The map $\hp: \bbR^n\times\bbR^n\to\bbR^{n\times n}$ given by
		\begin{equation}
			\hp(x,y) = \bbm \partialf_1\partials_1\h(x,y) & \cdots & \partialf_1\partialf_n\h(x,y)\\ \vdots & \ddots & \vdots\\  \partialf_n\partials_1\h(x,y) & \cdots & \partialf_n\partialf_n\h(x,y) \ebm,
		\end{equation}
		is symmetric and positive semidefinite.
	\end{proposition}
	\begin{IEEEproof}
		In view of Proposition~\ref{prop:feature_map}, it is enough to show that there exists a feature map $\phi:\bbR^n\to\B(\bbR^n,\calH)$ such that $\hp(x,y) = \phi(x)\a \phi(y)$. To this end, let $\phi$ be such that
		\begin{equation}
			\phi(x)v = \partials\h(\cdot,x)\t v = \sum_{i=1}^{n} \partials_i\h(\cdot,x)v_i,
		\end{equation}
		for all $x,v\in\bbR^n$, where $v_i\in\bbR$ is the $i$-th entry of $v$. Note that $\phi(x)v\in\calH$ due to the first item in Proposition~\ref{prop:partial_reproducing}. Due to the second item, we have that
		\begin{equation}
			\inner{H}{\phi(x)v}_\calH = \sum_{i=1}^{n} \partial_iH(x)v_i =\inner{ \partial H(x)}{v}_{\bbR^n}
		\end{equation}
		for all $H\in\calH$ and $x,v\in\bbR^n$, hence $\phi(x)\a H = \partial H(x)$ for all $x\in\bbR^n$ and $H\in\calH$. This implies that
		\begin{equation}
			\phi(x)\a \phi(y)v = \sum_{i=1}^{n} \partialf\partials_i\h(x,y)v_i = \hp(x,y)v
		\end{equation}
		for all $x,y,v\in\bbR^n$, and, thus, $\hp(x,y) = \phi(x)\a \phi(y)$.
	\end{IEEEproof}

Proposition~\ref{prop:partial_Gram} shows that $\hp$ is a reproducing kernel. Note that $\H_{ij} = \hp(x_i,x_j)$ for all $i,j\in[N]$, hence $\H$ is the Gram matrix associated with $\hp$ and $x_i\in\bbR^n$, $i\in[n]$, This implies that $\H$ is positive semidefinite. Furthermore, we have the following result regarding the norm of $H$.
\begin{lemma}\label{lem:representer_H_norm}
	If $H\in\calH$ is of the form \eqref{eq:representer_H}, then 
	\begin{equation}\label{eq:representer_H_norm}
		\norm{H}_\calH^2 = c\t \H c.
	\end{equation}
\end{lemma}
\begin{IEEEproof}
	Let the feature map $\phi:\bbR^n\to\B(\bbR^n,\calH)$ be defined as in the proof of Proposition~\ref{prop:partial_Gram}. Note that we can write $H = \sum_{i=1}^{N} \phi(x_i)c_i$, which implies that
	\begin{align*}
		\inner{H}{H}_\calH &= \sum_{i,j=1}^{n} \inner{\phi(x_i)c_i}{\phi(x_j)c_j}_\calH = \sum_{i,j=1}^{n} c_i\t \H_{ij}c_j,
	\end{align*}
	because $\phi(x_i)\a\phi(x_j) = \hp(x_i,x_j) = \H_{ij}$ for all $i,j\in[N]$. It is easily seen that the latter is equivalent to \eqref{eq:representer_H_norm}.
\end{IEEEproof}

Finally, using \eqref{eq:representer_H_eval} and \eqref{eq:representer_H_norm}, we conclude that a minimizer of \eqref{eq:min_representer_H} can be obtained by solving
\begin{equation}
	\min_{c\in\bbR^{nN}} L(\H_1c, \dots, \H_Nc) + \beta c\t \H c,
\end{equation}
and substituting the resulting minimizer in \eqref{eq:representer_H}.

%\old{, we obtain the following corollary of Proposition~\ref{prop:representer_H}. 
%\begin{corollary}\label{cor:representer_H}
%	Problem \eqref{eq:min_representer_H} has a minimizer of the form \eqref{eq:representer_H}, where $c = \col(c_1,\dots,c_N)$ is a minimizer of
%	\begin{equation}
%		\min_{c\in\bbR^{nN}} L(\H_1c, \dots, \H_Nc) + \beta c\t \H c,
%	\end{equation}
%	and $\H = \col(\H_1,\dots,\H_N)$ is the partial derivative Gram matrix associated with $\h$ and $x_i\in\bbR^n$, $i\in[N]$.
%\end{corollary}}
%\old{\begin{remark}\label{rem:universal_diff}
%	The notion of universality, see Remark~\ref{rem:universal}, can be extended to include the (partial) derivatives of a map. This is known as \emph{differential universality} \cite{guella2022}. A kernel is differentially universal if any continuously differentiable function and its derivative can be uniformly approximated to arbitrary precision by a map from the reproducing kernel Hilbert space and its derivative, respectively. Conditions for differential universality can be found in \cite{guella2022}. They are generally difficult to verify, except in some special cases, such as the one where the kernel is radial \cite[Theorem~4.12]{guella2022}. In particular, it is known that the Gaussian kernel is differentially universal.
%\end{remark}}

\subsection{Pointwise nonnegative maps}\label{app:nonnegative}

In this subsection, we leverage the theory of RKHS's to model pointwise nonnegative maps. We also prove a representer theorem for a class of optimization problems of the form \eqref{eq:min_representer_G} where the decision variable is constrained to be a pointwise nonnegative map from this model class. The results in this section are an extension of the results in \cite{marteau-ferey2020}, where only scalar-valued maps are considered, and \cite{muzellec2022}, where only pointwise nonnegative \emph{and symmetric} maps are considered. We also refer to our previous work \cite{shali2024}, where we consider nonnegative input-output operators.

Consider the map $F:\bbR^{n}\to\bbR^{p\times p}$. It is easily seen that there exist maps $l_i, r_i:\bbR^n\to\bbR^p$, $i\in[p]$, such that
\begin{equation}
	F = \sum_{i=1}^{p} l_i(\cdot)r_i(\cdot)\t.
\end{equation}
Indeed, one can take $l_i(x)\in\bbR^p$ to be the $i$-th column of $F(x)$, and $r_i(x)\in\bbR^p$ to be the $i$-th standard basis vector. We can model $F$ via $l_i, r_i$, $i\in[p]$, which can be modeled as maps in an RKHS. Let $\calF$ be an RKHS with kernel $\f:\bbR^{n}\times\bbR^{n}\to\bbR^{p\times p}$ and feature map $\phi:\bbR^{n}\to\B(\bbR^{p},\calW)$, where $\calW$ is a Hilbert space. Recall that $\calF$ consists of maps of the form $\phi(\cdot)\a w$, where $w\in\calW$. Consequently, we can write $l_i = \phi(\cdot)\a u_i$ and $r_i = \phi(\cdot)\a v_i$ for some $v_i,u_i\in \calW$, $i\in[p]$, which results in
\begin{equation}
	F = \phi(\cdot)\a \left( \sum_{i=1}^{p} u_iv_i\a \right)\phi(\cdot).
\end{equation}
The latter can be compactly written as
\begin{equation}\label{eq:F}
	F = \phi(\cdot)\a Q \phi(\cdot),
\end{equation}
where $Q = \sum_{i=1}^{p} u_iv_i\a\in\B(\calW)$ has finite rank and is, thus, a Hilbert-Schmidt operator, i.e., $Q\in\HS(\calW)$. Note that $F$ is pointwise nonnegative if $Q$ is nonnegative since
\begin{equation}
	\inner{F(x)v}{v}_{\bbR^p} = \inner{Q\phi(x)v}{\phi(x)v}_\calW \geq 0
\end{equation}
for all $x\in\bbR^n$ and $v\in\bbR^p$.  Therefore, we can model a pointwise nonnegative map $F$ as in \eqref{eq:F} with $Q\in\HS(\calW)_+$.

This model class is highly expressive. In particular, if the feature map $\phi$ is universal, then any continuous pointwise nonnegative map $F$ can be approximated arbitrarily well on compact subsets of $\bbR^n$. This has already been established for pointwise nonnegative scalar-valued maps in \cite{marteau-ferey2020} and pointwise symmetric and nonnegative matrix-valued maps in \cite{muzellec2022}. Extending these results to the case at hand, i.e., without pointwise symmetry, is fairly straightforward but outside of the scope of this paper. Instead, we prove yet another representer theorem using the following technical lemma.
\begin{lemma}\label{lem:representer_F}
	Let $x_i\in\bbR^{n}$, $i\in[N]$. For all $Q\in\HS(\calW)$, there exist $M_{ij}\in\bbR^{p\times p}$, $i,j\in[N]$, such that the operator
	\begin{equation}\label{eq:representer_Q}
		 \bar Q = \sum_{i,j=1}^{N} \phi(x_i)M_{ij} \phi(x_j)\a
	\end{equation}
	satisfies $\phi(x_i)\a \bar Q\phi(x_i) = \phi(x_i)\a Q\phi(x_i)$ for all $i\in[N]$, and $\norm{\bar Q}_{\HS} = \norm{Q}_{\HS}$. Furthermore, if $Q$ is nonnegative, then $\bar Q$ is nonnegative and the block matrix
	\begin{equation}\label{eq:representer_F_M}
		M = \bbm M_{11}  & \cdots & M_{1N}\\ \vdots & \ddots & \vdots \\ M_{N1} & \cdots & M_{NN} \ebm
	\end{equation}
	is nonnegative.
\end{lemma}
\begin{IEEEproof}
	Let $\Pi\in\B(\calW)$ be the orthogonal projection onto the finite-dimensional subspace $\bar\calW = \sum_{i=1}^{N} \im \phi(x_i)\subset\calW$, and define $\bar Q = \Pi Q \Pi$. Note that $\Pi = \Pi\a$ and $\Pi \phi(x_i) = \phi(x_i)$ for all $i\in[N]$, hence $\phi(x_i)\a \bar Q\phi(x_i) = \phi(x_i)\a \bar Q\phi(x_i)$ for all $i\in[N]$. Moreover, the composition of a Hilbert-Schmidt operator and a bounded operator is also a Hilbert-Schmidt operator \cite[p.267]{conway2007}, hence $\bar Q\in\HS(\calW)$ and
	\begin{equation}
		\norm{\bar Q}_\HS \leq \norm{\Pi}^2 \norm{Q}_{\HS} = \norm{Q}_{\HS},
	\end{equation}
	where $\norm{\Pi} = 1$ because $\Pi$ is an orthogonal projection. As shown in \cite[Lemma~3]{shali2024}, there exist $M_{ij}\in\bbR^{p\times p}$, $i,j\in[N]$ such that \eqref{eq:representer_Q} holds. Furthermore, $\bar Q$ and the block matrix in \eqref{eq:representer_F_M} are nonnegative if $Q$ is nonnegative. 
\end{IEEEproof}

We obtain the following representer theorem for optimization problems over pointwise nonnegative maps.
\begin{proposition}\label{prop:representer_F}
   	Let $x_i\in\bbR^{n}$, $i\in[N]$. Suppose that $\gamma > 0$ and $L: (\bbR^n)^N\to\bbR_+$ is continuous. Then, the problem
   	\begin{equation}\label{eq:min_representer_F}
   		\begin{aligned}
   			\min \quad& L(F(x_1), \dots, F(x_N)) + \gamma\norm{Q}_{\HS}^2,\\
   			\text{s.t.}\quad & Q\in\HS(\calW)_+ \text{ and } F = \phi(\cdot)\a Q \phi(\cdot)
   		\end{aligned}
   	\end{equation}
   	has a minimizer of the form \eqref{eq:representer_Q}, where the block matrix in \eqref{eq:representer_F_M} is nonnegative. Furthermore, the map
	\begin{align}
		\bar F &= \sum_{i,j=1}^{N} \f(\cdot, x_i) M_{ij} \f(\cdot, x_j )\t,\label{eq:representer_F}
	\end{align}
	is such that $\bar F = \phi(\cdot)\a \bar Q \phi(\cdot)$.
\end{proposition}
\begin{IEEEproof}
	Let $\bar \calQ$ be the subset of operators of the form \eqref{eq:representer_Q}, where the block matrix in \eqref{eq:representer_F_M} is nonnegative. The latter implies that the operators in $\bar\calQ$ are nonnegative. With this in mind, let the map $c:\HS(\calW)_+\to\bbR_+$ be given by
	\begin{equation}
		c(Q) = L(F(x_1), \dots, F(x_N)) + \gamma\norm{Q}_{\HS},
	\end{equation}
	where $F = \phi(\cdot)\a Q \phi(\cdot)$. Let $Q_0\in\bar \calQ$ be arbitrary and let $\bar\calQ_0$ be the closed and bounded subset of operators $Q\in\bar\calQ$ such that $\gamma \norm{Q}_{\HS} \leq c(Q_0)$. Note that $\bar\calQ_0$ is compact because $\bar\calQ$ is a subset of a finite-dimensional subspace. Let
	\begin{equation}
		\bar Q_0 = \argmin_{Q\in\bar\calQ_0} c(Q),
	\end{equation}
	which exists because $c$ is continuous and bounded from below, and $\bar\calQ_0$ is compact. We claim that $\bar Q_0$ is a minimizer of \eqref{eq:min_representer_F}. Due to Lemma~\ref{lem:representer_F}, for all $Q\in\HS(\calW)_+$, there exists $\bar Q \in\bar\calQ$ such that $c(\bar Q) \leq c(Q)$.  If $c(\bar Q) \leq c(Q_0)$, then, in particular, $\gamma\norm{\bar Q}\leq c(Q_0)$, hence $\bar Q\in\bar\calQ_0$, and, thus, $c(\bar Q_0) \leq c(\bar Q)$ by definition of $\bar Q_0$. On the other hand, if $c(Q_0) < c(\bar Q)$, then $c(\bar Q_0) < c(\bar Q)$ because $Q_0\in\bar\calQ_0$ and, thus, $c(\bar Q_0) \leq c(Q_0)$. It follows that $c(\bar Q_0) \leq c(Q)$ for all $Q\in\HS(\calW)_+$, hence $\bar Q_0$ is a minimizer of \eqref{eq:min_representer_F}. Since $\bar Q_0 \in \bar\calQ$, it is indeed of the form \eqref{eq:representer_H}, which concludes the proof. Finally, \eqref{eq:representer_F} follows from the fact that $\phi(\cdot)\a \phi(x_i) = \f(\cdot,x_i)$ for all $i\in[N]$.
\end{IEEEproof}

Like the previous representer theorems, the representer theorem stated in Proposition~\ref{prop:representer_F} provides a tractable solution to \eqref{eq:min_representer_F}, where a minimizer can be obtained by substituting $Q$ of the form \eqref{eq:representer_Q} and solving for coefficient matrices $M_{ij}\in\bbR^{p\times p}$, $i,j\in[N]$. The resulting finite-dimensional problem can be compactly expressed using the Gram matrix. In particular, if $F$ is of the form \eqref{eq:representer_F}, then
\begin{equation}\label{eq:representer_F_eval}
	F(x_i) = \F_i M \F_i\t
\end{equation}
for all $i\in[N]$, where $M$ is given by \eqref{eq:representer_F_M} and
\begin{equation}\label{eq:gram_Fi}
	\F_i = \bbm \f(x_i,x_1) & \cdots & \f(x_i,x_N) \ebm
\end{equation}
is the $i$-th block row of the Gram matrix $\F = \col(\F_1,\dots,\F_N)$ associated with $\f$ and $x_i\in\bbR^n$, $i\in[N]$. We also have the following result regarding the Hilbert-Schmidt norm of \eqref{eq:representer_Q}.
\begin{lemma}\label{lem:representer_F_norm}
	If $Q\in\HS(\calW)$ is of the form \eqref{eq:representer_F}, then
	\begin{equation}\label{eq:representer_F_norm}
		\norm{Q}^2_{\HS} = \tr(\F M\F M\t),
	\end{equation}
	where $M$ is given by \eqref{eq:representer_F_M}.
\end{lemma}
\begin{IEEEproof}
	Let the operator $\Phi\in\B(\bbR^{pN}, \calW)$ be given by
	\begin{equation}
		\Phi(z) = \sum_{i=1}^{N} \phi(x_i)z_i
	\end{equation}
	where $z_i\in\bbR^p$, $i\in[N]$, and $z = \col(z_1,\dots,z_N)$. We can write $Q = \Phi M \Phi^\ast$, which is easily verified using the block operator notation
	\begin{equation}
		\Phi = \bbm \phi(x_1) & \cdots & \phi(x_N) \ebm.
	\end{equation}
	Let $w_i$, $i\in[k]$, be an orthonormal basis for the finite-dimensional subspace $\im \Phi\subset \calW$, and extend it to an orthonormal basis $w_i$, $i\in\bbN$, for the whole space $\calW$. Note that $\inner{\Phi c}{w_i} = 0$ for all $c\in\bbR^{pN}$ and $i > k$, hence $\Phi\a w_i = 0$ for all $i > k$. Consequently, we obtain
	\begin{equation}
 \norm{Q}_{\HS}^2= \sum_{i=1}^{\infty} \norm{Qw_i}_\calW^2 = \sum_{i=1}^{k} \norm{\Phi M \Phi\a w_i}^2_\calW
	\end{equation}
	Now, for all $i\in[k]$, there exists $c_i\in\bbR^{pN}$ such that $w_i = \Phi c_i$. Since $\inner{w_i}{w_j}_\calW = 0$ for all $i\neq j$, it follows that
	\begin{equation}
		\inner{\Phi c_i}{\Phi c_j} = \inner{c_i}{\F c_j} = \sinner{\F^\half c_i}{\F^\half c_j} = 0
	\end{equation}
	where we used the fact that $\Phi\a \Phi = \F$ is symmetric positive semidefinite, hence there exists a symmetric positive semidefinite $\F^\half\in\bbR^{pN\times pN}$ such that $\F^\half \F^\half = \F$. Similarly, since $\norm{w_i}^2_\calW = 1$ for all $i\in\bbN$, it follows that $\snorm{\F^\half c_i}_{\bbR^{pN}} = 1$ for all $i\in[k]$. In other words, the vectors $v_i = \F^\half c_i\in\bbR^{pN}$, $i\in[k]$, are orthonormal. Note that $k$ is the rank of $\Phi$, which equals the rank of $\F$ and, thus $\F^\half$. This implies that $v_i$, $i\in[k]$ form an orthonormal basis for the subspace $\im \F^\half\subset\bbR^{pN}$, which can be extended to an orthonormal basis $v_i\in\bbR^{pN}$, $i\in[nP]$, for the whole space $\bbR^{pN}$. Since $\F^\half$ is symmetric, this basis is such that $\F^\half v_i = 0$ for all $i >k$. 
	
	With this in mind, since $\Phi\a \Phi = \F = \F^\half \F^\half$, and $w_i = \Phi c_i$ and $v_i = \F^\half c_i$ for all $i\in[k]$, it follows that
	\begin{equation*}
		\norm{\Phi M \Phi\a w_i}^2_\calW = \inner{\Phi M \F c_i}{\Phi M \F c_i}_\calW = \snorm{\F^\half M \F^\half v_i}_{\bbR^{pN}}^2
	\end{equation*}
	for all $i\in[k]$. This implies that
	\begin{equation}
		\norm{Q}_{\HS} =  \sum_{i=1}^{k} \norm{\Phi M \Phi\a w_i}^2_\calW = \sum_{i=1}^{k} \snorm{\F^\half M \F^\half v_i}_{\bbR^{pN}}^2,
	\end{equation}
	Since $\F^\half v_i = 0$ for all $i > k$, we can write
	\begin{equation}
		\norm{Q}_{\HS}  = \sum_{i=1}^{nP} \snorm{\F^\half M \F^\half v_i}_{\bbR^{pN}}^2 = \snorm{\F^\half M \F^\half}^2_{\HS},
	\end{equation}
	Recall that the Hilbert-Schmidt norm of a matrix is simply the Frobenius norm, hence
	\begin{equation*}
		\snorm{\F^\half M \F^\half}^2_{\HS} = \tr(\F^\half M \F M\t \F^\half) = \tr(\F M \F M\t),
	\end{equation*}
	where we used the fact that $\tr(AB) = \tr(BA)$ for all matrices $A$ and $B$ of appropriate sizes.
\end{IEEEproof}

\begin{remark}\label{rem:tr_convex}
	It can be shown that 
	\begin{equation}
		\tr(\F M \F M\t) = \vec(M)\t (\F\otimes\F) \vec(M),
	\end{equation}
	where $\otimes$ is the Kronecker product. Since $\F$ is positive semidefinite, it follows that $\F\otimes \F$ is positive semidefinite, which immediately shows that $\tr(\F M \F M\t)$ is convex in $M$. Moreover, if $\F$ is invertible, then $\tr(\F M \F M\t)$ is strongly convex with convexity parameter $\lambda_{\min}(\F\otimes\F) = \lambda_{\min}(\F)^2$.
\end{remark}

Finally, using \eqref{eq:representer_F_eval} and \eqref{eq:representer_F_norm}, we conclude that a minimizer of \eqref{eq:min_representer_F} can be obtained by solving
\begin{equation*}
	\begin{aligned}
		\min_{M\in\bbR^{pN\times pN}_+} \quad& L(\F_1 M \F_1\t, \dots, \F_N M \F_N\t) + \gamma\tr(\F M \F M\t),
	\end{aligned}
\end{equation*}
and substituting the resulting minimizer in \eqref{eq:representer_F}.

%
%\old{As a consequence of Lemma~\ref{lem:representer_F_norm} and Remark~\ref{rem:tr_convex}, we obtain the following corollary of Proposition~\ref{prop:representer_F}.
%\begin{corollary}\label{cor:representer_F}
%	Problem \eqref{eq:min_representer_F} has a minimizer of the form \eqref{eq:representer_F}, where the block matrix $M$ in \eqref{eq:representer_F_M} is a minimizer of
%	\begin{equation*}
%		\begin{aligned}
%			\min_{M\in\bbR^{pN\times pN}_+} \quad& l(\F_1 M \F_1\t, \dots, \F_N M \F_N\t) + \gamma\tr(\F M \F M\t),
%		\end{aligned}
%	\end{equation*}
%	and $\F=\col(\F_1,\dots,\F_N)$ is the Gram matrix associated with $\f$ and $x_i\in\bbR^n$, $i\in[N]$.
%\end{corollary}}

\bibliographystyle{ieeetr}
\bibliography{../../references/kernel_refs}

\begin{thebibliography}{10}

\bibitem{vanderchaft2014}
A.~van~der Schaft and D.~Jeltsema, ``Port-{Hamiltonian} systems theory: An
  introductory overview,'' {\em Foundations and Trends in Systems and Control},
  vol.~1, no.~2, pp.~173--378, 2014.

\bibitem{jeltsema2009}
D.~Jeltsema and J.~M. Scherpen, ``Multidomain modeling of nonlinear networks
  and systems,'' {\em IEEE Control Systems Magazine}, vol.~29, no.~4,
  pp.~28--59, 2009.

\bibitem{doerfler2009}
F.~Dörfler, J.~K. Johnsen, and F.~Allgöwer, ``An introduction to
  interconnection and damping assignment passivity-based control in process
  engineering,'' {\em Journal of Process Control}, vol.~19, no.~9,
  pp.~1413--1426, 2009.

\bibitem{vanderschaft2017}
A.~J. van~der Schaft, {\em {$L_2$}-Gain and Passivity Techniques in Nonlinear
  Control}.
\newblock Springer International Publishing, 2017.

\bibitem{cherifi2019}
K.~Cherifi, V.~Mehrmann, and K.~Hariche, ``Numerical methods to compute a
  minimal realization of a port-{Hamiltonian} system.'' arXiv:1903.07042, 2019.

\bibitem{benner2020}
P.~Benner, P.~Goyal, and P.~{Van Dooren}, ``Identification of
  port-{Hamiltonian} systems from frequency response data,'' {\em Systems \&
  Control Letters}, vol.~143, p.~104741, 2020.

\bibitem{schwerdtner2021}
P.~Schwerdtner, ``Port-{Hamiltonian} system identification from noisy frequency
  tesponse data.'' arXiv:2106.11355, 2021.

\bibitem{morandin2023}
R.~Morandin, J.~Nicodemus, and B.~Unger, ``Port-{Hamiltonian} dynamic mode
  decomposition,'' {\em SIAM Journal on Scientific Computing}, vol.~45, no.~4,
  pp.~A1690--A1710, 2023.

\bibitem{ortega2024a}
J.-P. Ortega and D.~Yin, ``Learnability of linear port-{Hamiltonian} systems,''
  {\em Journal of Machine Learning Research}, vol.~25, no.~68, pp.~1--56, 2024.

\bibitem{kutz2016}
J.~N. Kutz, S.~L. Brunton, B.~W. Brunton, and J.~L. Proctor, {\em Dynamic Mode
  Decomposition: Data-Driven Modeling of Complex Systems}.
\newblock SIAM, 2016.

\bibitem{pillonetto2010}
G.~Pillonetto and G.~De~Nicolao, ``A new kernel-based approach for linear
  system identification,'' {\em Automatica}, vol.~46, no.~1, pp.~81--93, 2010.

\bibitem{pillonetto2014}
G.~Pillonetto, F.~Dinuzzo, T.~Chen, G.~{De Nicolao}, and L.~Ljung, ``Kernel
  methods in system identification, machine learning and function estimation: A
  survey,'' {\em Automatica}, vol.~50, no.~3, pp.~657--682, 2014.

\bibitem{dinuzzo2015}
F.~Dinuzzo, ``Kernels for linear time invariant system identification,'' {\em
  SIAM Journal on Control and Optimization}, vol.~53, no.~5, pp.~3299--3317,
  2015.

\bibitem{ljung2020}
L.~Ljung, T.~Chen, and B.~Mu, ``A shift in paradigm for system
  identification,'' {\em International Journal of Control}, vol.~93, no.~2,
  pp.~173--180, 2020.

\bibitem{paulsen2016}
V.~I. Paulsen and M.~Raghupathi, {\em An Introduction to the Theory of
  Reproducing Kernel {Hilbert} Spaces}.
\newblock Cambridge University Press, 2016.

\bibitem{micchelli2006}
C.~A. Micchelli, Y.~Xu, and H.~Zhang, ``Universal kernels,'' {\em Journal of
  Machine Learning Research}, vol.~7, no.~95, pp.~2651--2667, 2006.

\bibitem{lazar2024}
M.~Lazar, ``A universal reproducing kernel {Hilbert} space for learning
  nonlinear systems operators.'' arXiv:2412.18360, 2024.

\bibitem{pillonetto2018}
G.~Pillonetto and A.~Chiuso, ``Identification of stable linear systems via the
  sequential stabilizing spline algorithm,'' {\em IFAC-PapersOnLine}, vol.~51,
  no.~15, pp.~19--24, 2018.

\bibitem{grussler2017}
C.~Grussler, J.~Umenberger, and I.~R. Manchester, ``Identification of
  externally positive systems,'' in {\em IEEE Conference on Decision and
  Control}, pp.~6549--6554, 2017.

\bibitem{khosravi2019}
M.~Khosravi and R.~S. Smith, ``Kernel-based identification of positive
  systems,'' in {\em Proceedings of the IEEE Conference on Decision and
  Control}, pp.~1740--1745, 2019.

\bibitem{vanwaarde2022}
H.~J. van Waarde and R.~Sepulchre, ``Kernel-based models for system analysis,''
  {\em IEEE Transactions on Automatic Control}, vol.~68, no.~9, pp.~5317--5332,
  2023.

\bibitem{huo2026}
Y.~Huo, T.~Chaffey, and R.~Sepulchre, ``Kernel modeling of fading memory
  systems,'' {\em IEEE Transactions on Automatic Control}, vol.~71, no.~3,
  pp.~2046--2052, 2026.

\bibitem{shali2024}
B.~M. Shali and H.~J. van Waarde, ``Towards a representer theorem for
  identification of passive systems,'' in {\em Proceedings of the IEEE
  Conference on Decision and Control}, pp.~8760--8765, 2024.

\bibitem{ebrahimkhani2025}
S.~Ebrahimkhani and J.~Lataire, ``Kernel-based estimation of frequency response
  function of strictly passive systems,'' {\em IEEE Control Systems Letters},
  pp.~1--1, 2025.

\bibitem{marteau-ferey2020}
U.~Marteau-Ferey, F.~Bach, and A.~Rudi, ``Non-parametric models for
  non-negative functions,'' in {\em Proceedings of the 34th International
  Conference on Neural Information Processing Systems}, 2020.

\bibitem{muzellec2022}
B.~Muzellec, F.~Bach, and A.~Rudi, ``Learning psd-valued functions using kernel
  sums-of-squares.'' arXiv:2111.11306, 2022.

\bibitem{gevers2005}
M.~Gevers, ``Identification for control: From the early achievements to the
  revival of experiment design*,'' {\em European Journal of Control}, vol.~11,
  no.~4–5, pp.~335--352, 2005.

\bibitem{hu2023}
Z.~Hu, C.~De~Persis, and P.~Tesi, ``Learning controllers from data via
  kernel-based interpolation,'' in {\em Proceedings of the IEEE Conference on
  Decision and Control}, pp.~8509--8514, 2023.

\bibitem{huang2024}
L.~Huang, J.~Lygeros, and F.~Dörfler, ``Robust and kernelized data-enabled
  predictive control for nonlinear systems,'' {\em IEEE Transactions on Control
  Systems Technology}, vol.~32, no.~2, pp.~611--624, 2024.

\bibitem{dejong2025}
T.~de~Jong, S.~Weiland, and M.~Lazar, ``A kernelized operator approach to
  nonlinear data-enabled predictive control.'' arXiv:2501.17500, 2025.

\bibitem{wang2025}
Z.~Wang and F.~Forni, ``Passive nonlinear {FIR} filters for data-driven
  control.'' arXiv:2508.05279, 2025.

\bibitem{karniadakis2021}
G.~E. Karniadakis, I.~G. Kevrekidis, L.~Lu, P.~Perdikaris, S.~Wang, and
  L.~Yang, ``Physics-informed machine learning,'' {\em Nature Reviews Physics},
  vol.~3, no.~6, pp.~422--440, 2021.

\bibitem{cherifi2020}
K.~Cherifi, ``An overview on recent machine learning techniques for
  port-{Hamiltonian} systems,'' {\em Physica D: Nonlinear Phenomena}, vol.~411,
  p.~132620, 2020.

\bibitem{bertalan2019}
T.~Bertalan, F.~Dietrich, I.~Mezić, and I.~G. Kevrekidis, ``On learning
  {Hamiltonian} systems from data,'' {\em Chaos: An Interdisciplinary Journal
  of Nonlinear Science}, vol.~29, no.~12, 2019.

\bibitem{smith2024}
T.~Smith and O.~Egeland, ``Learning {Hamiltonian} dynamics with reproducing
  kernel {Hilbert} spaces and random features,'' {\em European Journal of
  Control}, vol.~80, p.~101128, 2024.

\bibitem{greydanus2019}
S.~Greydanus, M.~Dzamba, and J.~Yosinski, ``Hamiltonian neural networks,'' in
  {\em Proceedings of the 33rd International Conference on Neural Information
  Processing Systems}, Curran Associates Inc., 2019.

\bibitem{zhong2020}
Y.~D. Zhong, B.~Dey, and A.~Chakraborty, ``Symplectic {ODE-Net}: {L}earning
  {Hamiltonian} dynamics with control,'' in {\em International Conference on
  Learning Representations}, 2020.

\bibitem{holmsen2024}
S.~Holmsen, S.~Eidnes, and S.~Riemer-Sørensen, ``Pseudo-{Hamiltonian} system
  identification,'' {\em Journal of Computational Dynamics}, vol.~11, no.~1,
  pp.~59--91, 2024.

\bibitem{desai2021}
S.~A. Desai, M.~Mattheakis, D.~Sondak, P.~Protopapas, and S.~J. Roberts,
  ``Port-{Hamiltonian} neural networks for learning explicit time-dependent
  dynamical systems,'' {\em Physical Review E}, vol.~104, no.~3, p.~034312,
  2021.

\bibitem{beckers2022}
T.~Beckers, J.~Seidman, P.~Perdikaris, and G.~J. Pappas, ``Gaussian process
  port-{Hamiltonian} systems: Bayesian learning with physics prior,'' in {\em
  Proceedings of the IEEE Conference on Decision and Control}, pp.~1447--1453,
  IEEE, 2022.

\bibitem{moradi2025}
S.~Moradi, G.~I. Beintema, N.~Jaensson, R.~Tóth, and M.~Schoukens,
  ``Port-{Hamiltonian} neural networks with output error noise models.''
  arXiv:2502.14432, 2025.

\bibitem{cherifi2025}
K.~Cherifi, A.~El~Messaoudi, H.~Gernandt, and M.~Roschkowski, ``Nonlinear
  port-{Hamiltonian} system identification from input-state-output data,''
  2025.

\bibitem{dinh2012}
Q.~Tran~Dinh, S.~Gumussoy, W.~Michiels, and M.~Diehl, ``Combining
  convex–concave decompositions and linearization approaches for solving
  bmis, with application to static output feedback,'' {\em IEEE Transactions on
  Automatic Control}, vol.~57, no.~6, pp.~1377--1390, 2012.

\bibitem{guella2022}
J.~C. Guella, ``Operator-valued positive definite kernels and differentiable
  universality,'' {\em Analysis and Applications}, vol.~20, no.~04,
  pp.~681--735, 2022.

\bibitem{meyer1976}
R.~Meyer, ``Sufficient conditions for the convergence of monotonic mathematical
  programming algorithms,'' {\em Journal of Computer and System Sciences},
  vol.~12, no.~1, pp.~108--121, 1976.

\bibitem{shapiro1997}
A.~Shapiro, ``First and second order analysis of nonlinear semidefinite
  programs,'' {\em Mathematical Programming}, vol.~77, no.~1, pp.~301--320,
  1997.

\bibitem{micchelli2005}
C.~A. Micchelli and M.~Pontil, ``{On Learning Vector-Valued Functions},'' {\em
  Neural Computation}, vol.~17, no.~1, pp.~177--204, 2005.

\bibitem{kalnishkan2009}
Y.~Kalnishkan, ``An introduction to kernel methods,'' {\em Technical Report,
  available at https://cml.rhul.ac.uk/publications/files/tr0901.pdf}, 2009.

\bibitem{aronszajn1950}
N.~Aronszajn, ``Theory of reproducing kernels,'' {\em Transactions of the
  American Mathematical Society}, vol.~68, no.~3, pp.~337--404, 1950.

\bibitem{vanwaarde2021arxiv}
H.~J. van Waarde and R.~Sepulchre, ``Kernel-based models for system analysis.''
  arXiv:2110.11735, 2021.

\bibitem{zhou2008}
D.-X. Zhou, ``Derivative reproducing properties for kernel methods in learning
  theory,'' {\em Journal of Computational and Applied Mathematics}, vol.~220,
  no.~1, pp.~456--463, 2008.

\bibitem{conway2007}
J.~B. Conway, {\em A Course in Functional Analysis}.
\newblock Springer New York, 2007.

\end{thebibliography}
\begin{IEEEbiography}[{\includegraphics[width=1in, height=1.25in, clip, keepaspectratio]{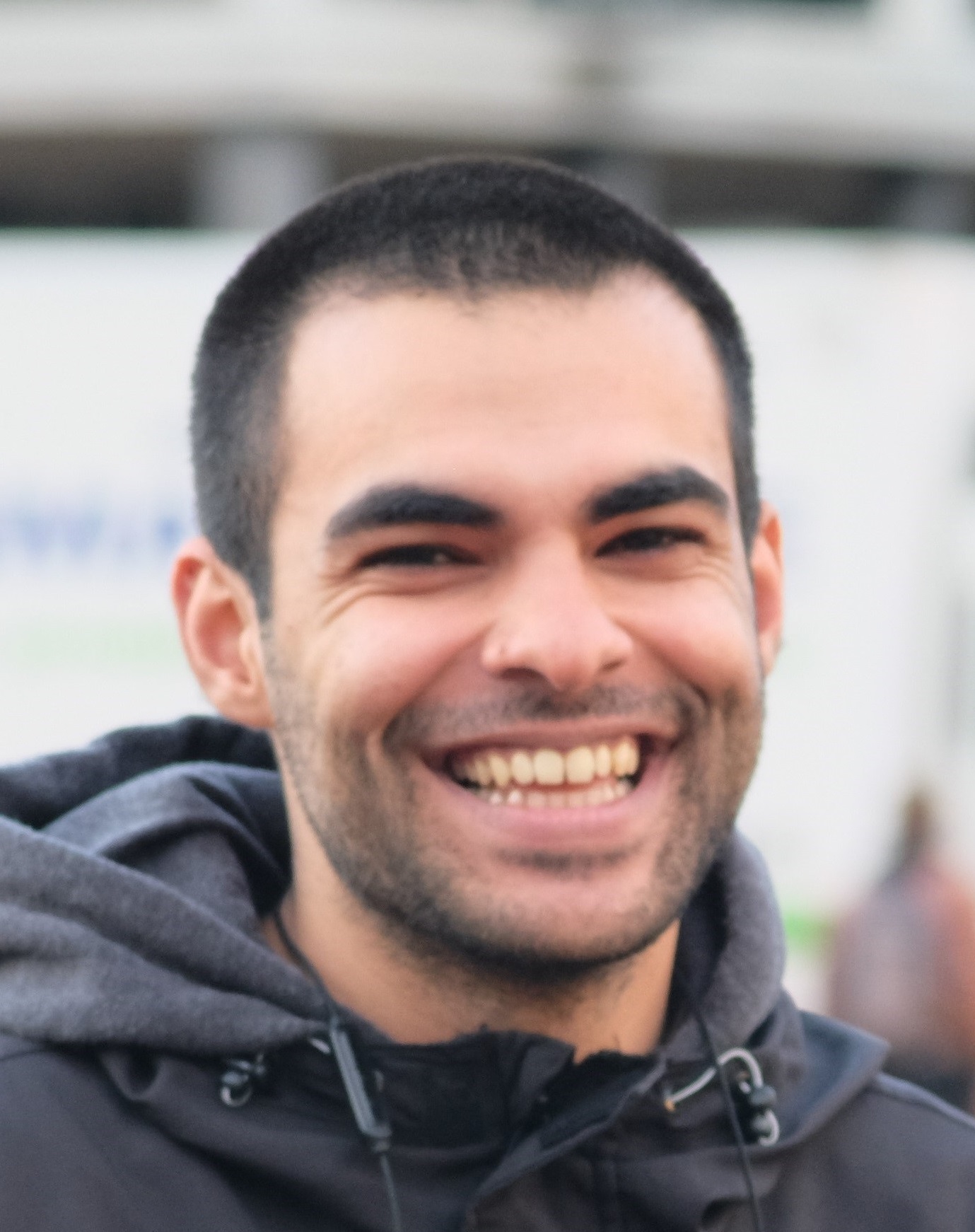}}]{Brayan M. Shali} received the M.Sc. degree (summa cum laude) and the Ph.D. degree in applied mathematics from the University of Groningen, Groningen, The Netherlands, in 2019 and 2023, respectively. He is currently a Postdoctoral Researcher with KU Leuven, Leuven, Belgium. His research interests include mathematical systems theory, contract-based design, networks of dynamical systems, and physics-informed modeling and identification. 
\end{IEEEbiography}

\begin{IEEEbiography}[{\includegraphics[width=1in,height=1.25in,clip,keepaspectratio]{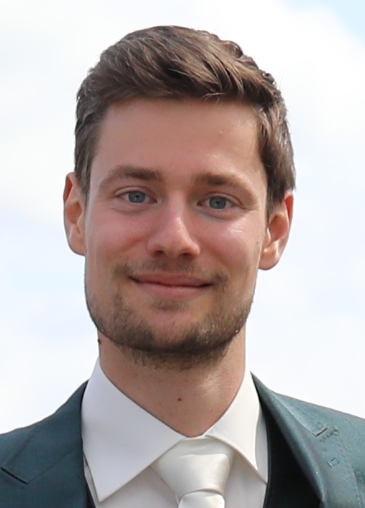}}]{Henk J. van Waarde} is an assistant professor in the Bernoulli Institute for Mathematics, Computer Science and Artificial Intelligence at the University of Groningen in The Netherlands. During 2020-2021 he was a postdoctoral researcher, first at the University of Cambridge, UK, and later at ETH Zürich, Switzerland. He obtained the Ph.D. degree \emph{cum laude} in Applied Mathematics from the University of Groningen in 2020. He was also a visiting researcher at University of Washington, Seattle in 2019-2020. His research interests include learning and data-driven control, system identification and identifiability, networks of dynamical systems, and robust and optimal control. Dr. van Waarde is the recipient of the 2025 SIAM Activity Group on Control and Systems Theory Prize. He serves as an Associate Editor of the IEEE Control Systems Letters.
\end{IEEEbiography}
\end{document}